\documentclass[reqno]{amsart}

\usepackage{mathrsfs}
\usepackage{amsmath,amstext,amsfonts,verbatim}
\usepackage [dvips]{graphics}
\usepackage[centertags]{amsmath}
\usepackage{amsfonts}
\usepackage{amssymb}
\usepackage{amsthm}
\usepackage{newlfont}
\usepackage{stmaryrd}
\usepackage[]{titletoc}

\newlength{\defbaselineskip}
\newcommand{\setlinespacing}[1]%
           {\setlength{\baselineskip}{#1 \defbaselineskip}}

\theoremstyle{plain}
\newtheorem{thm}{Theorem}[section]
\newtheorem{cor}[thm]{Corollary}
\newtheorem{lem}[thm]{Lemma}
\newtheorem{prop}[thm]{Proposition}
\newtheorem{exam}[thm]{Example}
\newtheorem{rem}[thm]{Remark}

\newcommand{\C}{\mathbb{C}}
\newcommand{\om}{\overline{\Omega}}

\makeatletter\@addtoreset{equation}{section} \makeatother
\begin{document}
\title {Multiplication operators on the Bergman space of bounded domains} 
 \author{Hansong Huang }
 \address{ School of Mathematics, East China
University of Science and Technology, Shanghai, 200237, China}
\email{hshuang@ecust.edu.cn}
 \author{Dechao Zheng }
 \address{Department of  Mathematics,  Vanderbilt  University,   Nashville, TN, 37240, United States }
 \email{dechao.zheng@vanderbilt.edu}
\date{}
 \maketitle \noindent\textbf{Abstract:}
  In this paper we study multiplication operators   on Bergman spaces of
 high dimensional bounded domains and those
 von Neumann algebras induced by them via the geometry of domains and function theory of their symbols.
 In particular, using local inverses and $L^2_a$-removability, we show that  for   a holomorphic proper map $\Phi =(\phi_1, \phi_2, \cdots , \phi_d)$ on a bounded domain $\Omega$ in  $\mathbb{C}^{d}$, the dimension   of the von Neumann algebra $\mathcal{V}^*(\Phi ,\Omega) $ consisting of
 bounded operators on the Bergman space $L_a^2(\Omega)$, which commute with both $ M_{\phi_j}$ and its adjoint $M_{\phi_j}^*$ for each $j$,
 equals the number of
 components of the complex manifold $\mathcal{S}_{\Phi }= \{(z,w)\in \Omega^2:  \Phi (z)=\Phi (w),\, z\not\in \Phi ^{-1}(\Phi  (Z))\},$
 where $Z$  is the zero variety of the Jacobian $J\Phi $ of $ \Phi.$ This extends the main result in \cite{DSZ} in high dimensional  complex domains.  Moreover we show that the von Neumann algebra $\mathcal{V}^*(\Phi ,\Omega) $ may not be abelian in general although   Douglas,  Putinar and  Wang  \cite{DPW}  showed that
 $\mathcal{V}^*(\Phi ,\mathbb{D})$ for the unit disk $\mathbb{D}$ is abelian.

\section{Introduction}

Let $\Omega$   be a bounded domain in the $d$-dimensional complex space $\mathbb{C}^d$, and
$dV$  the  Lebesgue measure on $\Omega$.  The Bergman space $L_a^2(\Omega)$
is the Hilbert space consisting of  all  holomorphic functions on $\Omega$ which are square
 integrable with respect to $dV$.
   For a bounded holomorphic function $ \phi  $ on   $\Omega$,
 let $M_{\phi}$ be the multiplication operator with the symbol $\phi $   on  $L_a^2(\Omega)$, given by
$$M_{\phi} f=\phi  f, \, f\in L_a^2(\Omega).$$
For a tuple $\Phi =(\phi_1, \cdots , \phi_n)$,
 let $\mathcal{V}^*(\Phi ,\Omega) $ denote the von Neumann algebra consisting of
 bounded operators on $L_a^2(\Omega)$ which commute with both $ M_{\phi_j}$ and its adjoint $M_{\phi_j}^*$ for each $j$.
 In fact,  as the range of an orthogonal projection  in $\mathcal{V}^*(\Phi ,\Omega) $
 must be a  joint reducing subspaces of $\{ M_{\phi_j}  :1\leq j \leq n \}$,  and vice versa,   we have  a natural correspondence between orthogonal projections in $\mathcal{V}^*(\Phi ,\Omega) $ and
   joint reducing subspaces of $\{ M_{\phi_j}  :1\leq j \leq n \}$ \cite{DSZ,GH1,GH4}.
   Many people have made  investigations on commutants, reducing subspaces of
    multiplication operators and von Neumann algebras induced  by those operators for the single-variable
    case \cite{CW,DPW,GH1,GH2,GH3,GH4,HZ1,HZ2,SZZ,T1,T2,Zhu}
  and some   multi-variable cases have been studied in  \cite{BDGS,DH,GW,Ti,WDH}.

In this paper we study the von Neumann algebra  $\mathcal{V}^*(\Phi ,\Omega) $ via studying the geometric properties of  $\Omega$, analytic properties of $\Phi$ and $L_a^2$-removable sets.

 Given   two  domains $\Omega$ and $\Omega^{\prime}$ in $ \mathbb{C}^d$, recall that
 a holomorphic map $\Psi: \Omega\to \Omega^{\prime}$ is said
 to be
  \emph{a proper map} if
  for each compact subset $K$ of $\Omega'$, $\Psi^{-1}(K)$
is compact. A holomorphic map $\Psi$ on $\Omega$ is said to be
proper if $\Psi(\Omega)$ is open and  $\Psi: \Omega \to  \Psi
(\Omega)$ is proper. The following theorem is one of our main results.

\begin{thm}\label{mainproper} Suppose $\Phi   $ is  a holomorphic proper map   on $ \Omega $.
The dimension of $\mathcal{V}^*(\Phi ,\Omega)$ equals the number of
	components of $\mathcal{S}_{\Phi }$. Here $\mathcal{S}_{\Phi }$ is a complex manifold in $\mathbb{C}^{2d}$:
	$$\mathcal{S}_{\Phi }= \{(z,w)\in \Omega^2:  \Phi (z)=\Phi (w),\, z\not\in \Phi ^{-1}(\Phi  ({Z}))\},$$
	where $Z$  denotes the zero variety of the Jacobian $J\Phi $ of $ \Phi.$
\end{thm}

 On the unit disk $\mathbb{D}$, a holomorphic proper map
$\phi $ is a finite Blaschke product up to composition of a biholomorphic map and  thus Theorem \ref{mainproper} generalizes the main result on the unit disk in  \cite{DSZ}. In \cite{DPW},  Douglas,  Putinar and  Wang   proved that
$\mathcal{V}^*(\phi ,\mathbb{D})$ is abelian.
In Section \ref{six},  we will give examples that   $\mathcal{V}^*(\Phi ,\Omega)$ is not abelian  for some $\Phi, ~\Omega$ in high dimensional spaces (for details, see Example \ref{exam6.1}). We will give an example that  $\Phi   $ is not a holomorphic proper map   on $ \Omega $ but $\mathcal{V}^*(\Phi
,\Omega)$  is of infinite dimension in Example \ref{interior}.

	We will show that the dimension of $\mathcal{V}^*(\Phi ,\Omega)$ equals  the number of
	equivalent classes of   local inverses of $\Phi $. Local inverses and admissible local inverses (for its definition see the end of Section 2) will play an important role in our study. In the
proof of Theorem \ref{mainproper} in Section \ref{sectfif}, we will show that all local inverses of a holomorphic  proper map must be admissible.

 We say that $\mathcal{V}^*(\Phi ,\Omega) $ is  trivial
if $\mathcal{V}^*(\Phi ,\Omega) =\mathbb{C}I$.  This is equivalent to that
$\{ M_{\phi_j}  :1\leq j \leq n \}$ has no nontrivial  joint reducing subspace.
 The following result is a consequence of Theorem \ref{mainproper}, which was contained in Theorems 2.2 and 2.4 in \cite{Gh}.
\begin{thm}\cite{Gh} Suppose $\Phi  :\Omega \to \Omega'  $ is a holomorphic  proper map.
	Then           \label{41}
	$ \mathcal{V}^*(\Phi ,\Omega ) $ is nontrivial if and only if $\Phi $ is not   biholomorphic.
\end{thm}
Since a holomorphic  proper map is  onto \cite[Proposition 15.1.5]{Ru2},  it is biholomorphic
if and only if it is univalent. In 
    multi-variable cases ``nontrivial" holomorphic  proper maps may arise from
polynomials. For example, both $(z_1+z_2,z_1z_2)$ and $(z_1^2
z_2^4,z_1^2  +z_2^4 )$ are holomorphic proper maps on the bidisk
$\mathbb{D}^2$. On the other hand,  $(z_1+z_2,z_1z_2)$
is also a  holomorphic proper map on the unit ball $\mathbb{B}_2$ of $\mathbb{C}^2$, but $(z_1^2
z_2^4,z_1^2  +z_2^4 )$ is not in $\mathbb{B}_2$. More details about these maps will be discussed
in Section \ref{six}.

 For  a local inverse, let  $[\rho]$ denote  the equivalent class of $\rho$,
  which consists of all  analytic  continuations of $\rho$.
We define an operator $\mathcal{E}_{[\rho]}$ on the
Bergman space $L_a^2(\mathbb{D})$ to be
\begin{equation}
\mathcal{E}_{[\rho]}h =\sum_{\sigma\in [\rho]} h \circ \sigma \, J\sigma ,  h \in L_a^2(\mathbb{D}),
 \label{beg}\end{equation}  where the sum involves only finitely many terms
 and $J \sigma $  denotes the determinant of the Jacobian of $\sigma$ for a
holomorphic map $\sigma.$ The local inverses
 can be defined for a holomorphic map $F:\Omega\to \C^d$.  To extend the definition of the operator $
  \mathcal{E}_{[\rho]}$ in (\ref{beg})  to high dimensional domains, we need to study $L^2_a$-removable sets.
 Indeed,    Douglas, Sun and the second author \cite{DSZ}  have proved that
    $\mathcal{V}^*(\phi , \mathbb{D})$ equals the linear span of $\mathcal{E}_{[\rho]}$ where $\rho$ runs over
all  admissible local inverses of each finite  Blaschke product $\phi$, the dimension   of $\mathcal{V}^*(\phi, \mathbb{D})$ equals the number of components of a Riemann surface
  $\mathcal{S}_\phi$ defined as in Theorem \ref{mainproper}.

  For general holomorphic maps, we have  the following theorem.


 \begin{thm}\label{thm1.1}  Suppose  that  $\Omega$ is a bounded domain in $\C^d$. If
 the interior of  $\overline{\Omega}$ equals $\Omega$, and 
 $\Phi  :\Omega\to \C^d$ is  holomorphic  on a neighborhood of $\overline{\Omega}$  such that
 the image of  $\Phi  $ has an interior point,
 then  $\mathcal{V}^*(\Phi ,\Omega)$  is a finite dimensional von Neumann algebra. Moreover
  $\mathcal{V}^*(\Phi ,\Omega)$ is generated by   \label{thm11}
  $\mathcal{E}_{[\rho]}$, where $\rho$ run over admissible local inverses  of $\Phi.$
    \end{thm}

\begin{rem}
\begin{itemize}
	\item[(1)]   If   the image of  $\Phi  $ has no  interior point, Example \ref{interior} shows that $\mathcal{V}^*(\Phi
	,\Omega)$  may be of infinite dimension.
	\item [(2)] The condition on $\Omega$ in the above theorem is equivalent to that the boundary $\partial \Omega$ of $\Omega$ is not "huge", i.e,  each point of $\partial
	\Omega $ is a limit point of the interior points of  $\mathbb{C}^d\setminus\!\Omega$. This observation is contained in the following equivalent statements:
	 \begin{itemize}
		\item[(A)]  The interior of $\overline{\Omega}$ equals
		$\Omega$;
		\item [(B)] For each point $\zeta\in \partial
		\Omega $ and any open neighborhood  $\mathcal{O}(\zeta)$ of
		$\zeta$, $\mathcal{O}(\zeta)\setminus\!\Omega$ contains an open ball.
	\end{itemize}
	 (B) immediately implies that   domains with $C^1$-boundary,   star-shaped
	domains, circled bounded domains, convex
domains (bounded symmetric domains), and  strictly pseudoconvex domains satisfy the condition
	on $\Omega $ in Theorem \ref{thm1.1}.
	
 To get the above equivalence, first suppose that (B) holds.  If (A)
  fails, then there is a point  $\lambda$ in the interior of
  $\overline{\Omega}$ and in
  $\partial \Omega $. So there is a neighborhood $\mathcal{N}$ of
  $\lambda$
  such that $\mathcal{N}\subseteq \overline{\Omega}$. Thus  the point
  $\lambda$ is not a limit point of
   $\mathbb{C}^d\setminus\!\overline{\Omega}.$ This contradicts (B) that
    $\lambda$ can
  be    approximated by a sequence of interior points
  in $\mathbb{C}^d\setminus\!\Omega.$

  Conversely,  suppose (A) holds.  If (B) fails, then
 for some point $\zeta\in \partial
\Omega $, there is a neighborhood $\mathcal{O}(\zeta)$ of
  $\zeta$ such that $\mathcal{O}(\zeta)\setminus\!\Omega$ contains no open ball. Thus    there is no interior point of
   $\mathcal{O}(\zeta)\setminus\!\Omega$ in $\mathcal{O}(\zeta)$, and so we have that $\mathcal{O}(\zeta)$ is contained in  $\overline{\Omega}.$
  This gives that $\zeta$ is an interior point of $\overline{\Omega}$, hence by (A), $\zeta$ is an interior point of $\Omega$. But it  contradicts that $\zeta$ is in the boundary $ \partial \Omega$.  Thus $(B)$
  holds.

\item[(3)]  Although a pseudo-convex domain with $C^1$ boundary
    satisfies the condition in  Theorem \ref{thm1.1},
  the following example shows that some pseudo-convex domains may not enjoy (A).
\begin{exam} \label{examdom}
The domain
 $\mathbb{D}\setminus\![0,1]$ is   a pseudo-convex domain whose interior of its closure does not equal itself as it is conformally equivalent to the unit disk $\mathbb{D}$.

Let  $\Omega=(\mathbb{D}\setminus\![0,1])\times\mathbb{D} $. Since
the biholomorphic image of a pseudo-convex domain is pseudo-convex and
  $\Omega$ is biholomorphic to $\mathbb{D}^2$,  $\Omega$ is pseudo-convex. But the interior of $\overline{\Omega}$
  is not $\Omega.$
\end{exam}

\item[(4)] In \cite[Theorem 6.1]{Ti}, for bounded smooth pseudoconvex domains
with the following two assumptions, Tikaradze showed that $\mathcal{V}^*(\Phi ,\Omega)$ is  isomorphic to
the algebra of locally constant functions on $W_\Phi$ under convolution product where $W_\Phi$ is a subset of the complex manifold:
$$\mathcal{S}_{\Phi }= \{(z,w)\in \Omega^2:  \Phi (z)=\Phi (w),\, z\not\in \Phi ^{-1}(\Phi  (\overline{Z}))\},$$
and $W_\Phi$ is defined in \cite[Definition 3.2]{Ti}.

{\bf Assumption 1.} For any $z\in \partial \Omega$,   $A^\infty (\Omega)\cap I(z)$ is dense in $L^{2}_{a}(\Omega)$ where $A^\infty (\Omega)$ is the set of all holomorphic functions on $\Omega$
which are $C^\infty$
smooth on $\overline{\Omega},$ and
$ I(z)$ denotes the ideal of holomorphic
functions on $\overline{\Omega}$
which vanish on $z$.

{\bf Assumption 2.}  There exists a nonzero function $g\in A^\infty (\Omega)$ such that $g$ vanishes on $\Phi ^{-1}(\Phi  (\overline{Z}))$ where $Z$ the zero variety of the Jacobian $J\Phi $ of $ \Phi$.

 As a domain $\Omega$ in $\mathbb{C}^d$ with $C^1$-boundary has the property that the interior of $\overline{\Omega}$ equals
 $\Omega$  \cite[p. 117]{Kr}, \cite[p. 52]{Ran},
 Theorem  \ref{thm1.1} suggests that for connected bounded  domains with $C^1$-boundary, Theorem 6.1 in \cite{Ti} may  hold.
 \end{itemize}
\end{rem}

 Theorem \ref{thm11}  also gives the following criterion when
$\mathcal{V}^*(\Phi ,\Omega)$ is nontrivial.
\begin{cor} \label{34} Suppose  that  $\Omega$ is a bounded domain. If
 the interior of  $\overline{\Omega}$ equals $\Omega$, and 
 $\Phi  :\Omega\to \C^d$ is  holomorphic  on a neighborhood of $\overline{\Omega}$  such that
 the image of  $\Phi  $ has an interior point,
  then $\mathcal{V}^*(\Phi ,\Omega)$ is nontrivial if and only if  \label{exist}
 there exists   an admissible local inverse of $\Phi $ distinct from the identity map.
  \end{cor}

On the complex plane $\mathbb{C}$,   the assumption in Corollary \ref{exist}  says that $\Phi$ is nonconstant and holomorphic over
 $\overline{\Omega}$.
  On the unit disk, Corollary \ref{exist}
is implicitly contained in \cite{T1}. On a polygon  $\Omega$, $\mathcal{V}^*(\Phi ,\Omega)$
is often trivial and the nontrivialness of $\mathcal{V}^*(\Phi ,\Omega)$
implies some geometric property of   $\Omega$ \cite{HZ1,HZ2}.

 To get Theorem \ref{thm1.1}, we  need to study $L_a^2$-removable sets.
Let $A= \Phi ^{-1}(\Phi  (\overline{Z})),$
 where $Z$ is the zero variety of the Jacobian $J\Phi $ of $ \Phi.$   $A$ is a relative closed subset of $\Omega $ but ``small". For each function $h$ in $L_a^2(\Omega)$,  $ \mathcal{E}_{[\rho]}h$ is
  a well-defined holomorphic function on $\Omega\setminus\! A$.
   The function $ \mathcal{E}_{[\rho]}h$ lies in $L_a^2(\Omega\setminus\!A)$.
  In the case of the unit disk,
$A$ is a discrete subset of $\mathbb{D}$, and thus by complex
analysis one  easily gets that every function $f\in
L_a^2( \mathbb{D}\setminus\!A)$ can extend analytically to $\mathbb{D}$ and the extension belongs to $L_a^2( \mathbb{D}).$
However, if $\Omega$ is a higher dimensional domain,      $A$
 is more complicated than a zero variety, and hence $A$ is far from
discrete. We must show that each function in
$L_a^2(\Omega\setminus\!A)$ extends holomorphically to $\Omega.$ This naturally
involves a  problem of removability. Precisely, a relatively closed
subset $\mathcal{E}$ of $\Omega$ is called \emph{$L_a^2$-removable} in
$\Omega$ if
  each function in $L_a^2(\Omega\setminus\! \mathcal{E})$ extends holomorphically to a function
  in $L_a^2(\Omega )$ \cite{Bj,Car}.  An improved version of the Riemann Removable Singularity Theorem \cite{Be,Bo} states that   a zero variety of a domain $\Omega$ in $\mathbb{C}^d$ is   $L_a^2$-removable  in $\Omega $.

  For a subset  $E$ of $\Omega$, $\overline{E}$   denotes
the closure of $E$ in $\overline{\Omega}$.  A subset $E$ of $\Omega$ is called a zero variety of $\Omega$ if there is a nonconstant holomorphic function $f$ on $\Omega$ such that
$$E=\{z\in \Omega: f(z)=0\}.$$
A subset $E$ of $\Omega$ is called a    {local} zero variety if
 for each point  $\lambda \in {\Omega}$, there are an open neighborhood $U$ of $\lambda$ and a {nonconstant} holomorphic function $h_{\lambda}$ on $U$ such that
\begin{equation}\label{localzero}
U\cap \overline{E}\subset \{ z \in U\cap \overline{\Omega}: h_{\lambda}(z)=0\}.
\end{equation}
Clearly, every analytic set in \cite{Ch} of $\Omega$ is a local zero variety.
A local zero variety $E$ of $\Omega$ is said to be {\sl good}      if (\ref{localzero}) holds for each point $\lambda$ in $\overline{\Omega}$.
  For example, if $h$ is a
nonconstant holomorphic function over $\overline{\Omega}$, then
$\{z\in \Omega : h(z)=0\}$ is a  good {local} zero variety of $\Omega.$ For
   a bounded planar domain $\Omega$, 
    a subset of $\Omega$ is a good local zero variety
 if and only if it is  a finite set. To see this, we need to show
 that  a good local zero variety $E$ of  a bounded planar domain $\Omega$ is  a finite subset.
 Otherwise, if $E$ were
   an infinite subset of  $\Omega$, then $E$
    would contain one  accumulation point $w$ on $\overline{\Omega}$.
     By definition, there is a unit disk $U$ centered at $w$ such that
     $U\cap \overline{E}\subset \{ z \in U\cap \overline{\Omega}: h (z)=0\} $
     for some holomorphic function $h$ over $U.$ Then $h$ would be identically zero
     as  $w$ were an accumulation point,
     which is a contradiction.

  We need the following theorem about $L^2_a$-removable sets, which generalizes  the improved version of the Riemann Removable Singularity Theorem \cite{Be,Bo}.
 For a   map $F: \Omega\to \mathbb{C}^d$
 holomorphic over $\overline{\Omega}$,
   let
 $$F^{-1}(F(\overline{E})) =\{z\in \Omega: F(z) \in F(\overline{E})\}. $$

\begin{thm}\label{removst2}  Let  $\Omega$ be a domain in  $\mathbb{C}^d.$  Suppose that a local zero variety $E $  of $\Omega$ is   good    and
 $F: \Omega\to \mathbb{C}^d$ is holomorphic on $\overline{\Omega}$ such
that the image of $F$ contains an interior point.
 Then
   $  F^{-1}(F(\overline{E}))$  is   $L_a^2$-removable in $\Omega$.
\end{thm}

Theorem \ref{removst2} immediately gives that  each good {local} zero variety of $\Omega$
is $L_a^2$-removable in $\Omega$.


\vskip1mm
This paper is arranged as follows. Section \ref{sect2} contains some preliminaries such as the notion of admissible local inverse,
 and some properties of holomorphic proper maps.
In Section \ref{sect3}  a new approach is presented to prove the improved version of the Riemann Removable Singularity Theorem \cite{Be,Bo},  and using some ideas of the approach, we will present the proof of   Theorem \ref{removst2}.
     In Section \ref{sect4}
  we will present the proof of Theorem \ref{thm11} and  show 
   how admissible local inverse
  matches   $L_a^2$-removable property. Section \ref{sectfif} contains
   the proofs of Theorems  \ref{mainproper} and \ref{41}.
   In Section \ref{six}  we  will provide many examples of $\mathcal{V}^*(\Phi,\Omega) $
   and  show that $\mathcal{V}^*(\Phi,\Omega) $ has fruitful structures.



\section{\label{sect2}Preliminaries}

 In this section, we will introduce some  notations and give definitions of local inverses and the admissible local inverse  in multi-variable case.  Some properties of holomorphic proper maps in \cite{Ru2}  will be stated as they will be used later.

 Let $F:\Omega\to \mathbb{C}^d$ be a holomorphic map. Let    $Z$
be the set $Z(JF)$ of the zeros of the determinant of the Jacobian $JF$ of the map $F$.
The image $F(Z)$ is
 called \emph{the critical set} of $F$, each point in $F(Z)$ is called a critical value, and
each point in $F(\Omega)\setminus\!F(Z)$ is called a regular point.

Let $\Omega$ and $\Omega^{\prime}$  be two  domains in $ \mathbb{C}^d$. In the introduction,
it is mentioned that
 a holomorphic map $\Psi: \Omega\to \Omega^{\prime}$ is called  a proper map if
  for each compact subset $K$ of $\Omega'$, $\Psi^{-1}(K)$
is compact. 
 For example, a holomorphic proper function from the unit
disk $\mathbb{D}$ to $\mathbb{D}$ is just a finite Blaschke
product. More generally, a holomorphic map $\Phi $ from the polydisk
$\mathbb{D}^d$ to $\mathbb{D}^d$ is proper if and only if
  $$\Phi (z)=(\phi _1(z_1),\cdots, \phi _d(z_d))$$
   up to a permutation of coordinates where
 $\phi _j(1\leq j \leq d)$ are finite Blaschke products \cite[Theorem 7.3.3]{Ru1}.
A holomorphic proper map is   a surjective open map \cite[Proposition
15.1.5]{Ru2}.  The following result is contained in Theorem 15.1.6 in
\cite{Ru2}.
\begin{thm} \label{prop}  Suppose  $F: \Omega \to  \mathbb{C}^d$ is a holomorphic function
and for each point $w\in \C^d$, $F^{-1}(w)$ is compact. Then $F$ is
an open map.
\end{thm}
Moreover a  holomorphic proper map is   an $m$-folds map and
its critical set enjoys  the following good properties
\cite[Theorem 15.1.9]{Ru2}.
\begin{thm} For two domains $\Omega$ and $\Omega'$ in $\mathbb{C}^d$, suppose $F:\Omega\to \Omega'$ is a
holomorphic proper function. Let $\sharp (w)$ denote the number of elements in
$F^{-1}(w)$ with $w\in \Omega'$.  \label{variety} Then the following hold:
\begin{itemize}
  \item[$(1)$] There is an integer $m$ such that $\sharp (w)=m$ for all regular values $w$ of $F$
  and  $\sharp (w')<m$ for all critical values $w'$ of $F$;
  \item [$(2)$] The  critical set of $F$ is a zero variety in $\Omega'.$
\end{itemize}
\end{thm}

To study local inverses, we need some notations about analytic continuation \cite[Chapter
16]{Ru3}.  A function element  is an ordered pair $(f,D)$, where $D$
is a simply-connected domain and $f$ is a holomorphic function on
$D$. Two  function elements $(f_0,D_0)$  and $(f_1,D_1)$ are  called
direct continuations if $D_0\cap D_1$ is not empty and $f_0(z) $ equals $f_1(z) $
on $D_0\cap D_1$. A curve is
  a continuous map from $[0,1]$ into  $
\mathbb{C}^d$.
  Given a function element $(f_0,D_0)$ and  a curve    $\gamma$   with $\gamma(0)\in D_0$,
if there are  a partition of  $[0,1]$:
     $$0=s_0<s_1<\cdots <s_n=1$$
 and function elements $(f_j,D_j)(0\leq j \leq n)$ such that
   \begin{itemize}
  \item [1.]  $(f_j,D_j)$  and $(f_{j+1},D_{j+1})$
are direct continuation for all $j$ with\linebreak $ 0\leq j\leq n-1
$;
  \item[2.] $\gamma [s_j,s_{j+1}]\subseteq D_j(0\leq j\leq n-1)$ and $\gamma(1)\in D_n$,
       \end{itemize}then   $(f_n, D_n)$ is called \emph{an analytic continuation of $(f_0, D_0)$ along
     $\gamma$} ; and $(f_0, D_0)$
 is called to \emph{admit an analytic continuation along}  $\gamma$.
 In this case, we write   $f_0\sim  f_n$. Clearly, $\sim$ defines an equivalence and we write
$[f]$ for \emph{the equivalent class} of $f$.

Given a family $\{\phi_j\}$ of holomorphic maps  over $\Omega$, and a subdomain $\Delta$ of $\Omega $, a
holomorphic function $\rho: \Delta \to \Omega $ is said to be \emph{a
local inverse of} $\{\phi_j\}$ on $\Delta$ if $\phi_j\circ \rho
=\phi_j$ for all $j.$ In particular,  if $\Phi=(\phi_1,\cdots,
\phi_d)$, $\rho$
 is said to be a {\bf local inverse} of $\Phi$ if $\Phi\circ \rho=\Phi$. 

For a holomorphic map $\Phi:\Omega\to \C^d$, assume the image of $\Phi$ contains an interior point. Let
$$A=\overline{\Phi^{-1}\Big(\Phi(Z(J\Phi))\Big)}$$
where $Z(J\Phi)$ is the set of zeros  of the determinant of the Jacobian $J\Phi$ of $\Phi$.
A
local inverse $\rho$ of $\Phi
:\Omega\to \C^d$ is said to be  {\bf admissible} if      
  for each curve $\gamma$ in $\Omega\setminus\!A$, $\rho$ admits analytic
continuation with values in $\Omega$.    The notion of admissible  local inverse on the unit disk $\mathbb D$
was first introduced  by Thomson \cite{T1}  for   a discrete set $A$.

\section{\label{sect3} $L_a^2$-removable sets}

In this section  we will present the proof of Theorem \ref{removst2} and show that some sets slightly more general than a zero variety  enjoy
the $L_a^2$-removable property, which is  needed  in the proofs of
Theorems  \ref{mainproper} and \ref{thm11}.

For completeness using a different approach we present a proof of the following result in \cite{Be,Bo}.
Some ideas in the proof will be used in the proof of  Theorem  \ref{removst2}.

 \begin{thm} \label{removst}Suppose that  $E $ is a zero variety of a domain $\Omega$ in $\mathbb{C}^d$.
Then $E$ is   $L_a^2$-removable  in $\Omega $.
\end{thm}
 \begin{proof} Without loss of generality,  we may assume $d>1.$
  First let us consider the special case that $\Omega$ is the polydisk $\mathbb{D}^{d } $ and
 $E$ is equal to $\{z\in \Omega: z_d=0\}.$ Then $$\Omega\setminus\!E= \mathbb{D}^{d-1}\times (\mathbb{D}\setminus\!\{0\}).$$
We will show that holomorphic functions in
$L_a^2\big(\mathbb{D}^{d-1}\times (\mathbb{D}\setminus\!\{0\})\big)$ can extend
analytically to members in    $L_a^2(\mathbb{D}^d)$.
 Since $\{z^\beta: \beta\in  \mathbb{Z}_+^{d }\}$ is an orthogonal basis of   $L_a^2(\mathbb{D}^d)$,  each function $h$ in $L_a^2(\mathbb{D}^d)$   has the following power series expansion:
$$h(z)=\sum_{\beta  \in   \mathbb{Z}_+ ^{d }}c_\alpha z^\alpha , \, z\in \mathbb{D}^{d },$$
where $\sum_{\alpha\in \mathbb{Z}_+^d}  |c_\alpha|^2 \|z^\alpha\|^2 <\infty.$

Let $g$ be  in
$L_a^2\big(\mathbb{D}^{d-1}\times (\mathbb{D} \setminus\!\{0\})\big)$.
We can write in   Laurent series:
$$g(z) = \sum_{\alpha  \in   \mathbb{Z}_+ ^{d-1 }\times \mathbb{Z} }c_\alpha' z^\alpha, \, z\in \mathbb{D}^{d-1}\times (\mathbb{D}\setminus\!\{0\}).$$
Since \begin{eqnarray*}
 \|g\|^2_{L_a^2\big(\mathbb{D}^{d-1}\times (\mathbb{D} \setminus\!\{0\})\big)}&=&\int_{\mathbb{D}^{d-1}\times (\mathbb{D}\setminus\!\{0\})} |g(z)|^2dV(z)\\
 &=& \sum_{\alpha\in \mathbb{Z}_+^d}  |c_\alpha'|^2 \|z^\alpha\|^2+\sum_{\alpha\in
  \mathbb{Z}_+^{d-1}\times \mathbb{Z}_-}  |c_\alpha'|^2 \int_{\mathbb{D}^{d-1}\times (\mathbb{D} \setminus\!\{0\})}|z^\alpha |^2dV(z),
 \end{eqnarray*}
we have that $c_\alpha' =0$ for $\alpha$ in $\mathbb{Z}_+^{d-1}\times \mathbb{Z}_-$ and hence
$$\|g\|^2_{L_a^2\big(\mathbb{D}^{d-1}\times (\mathbb{D}\setminus\!\{0\})\big)}=\sum_{\alpha\in \mathbb{Z}_+^d}  |c_\alpha'|^2\|z^\alpha\|^2 .$$
 This   gives that
 $$g(z) = \sum_{\alpha  \in   \mathbb{Z}_+ ^{d } }c_\alpha' z^\alpha, $$
and thus  $g$ is in $ L_a^2( \mathbb{D}^d)$.

For general case, let   $E $ be a zero variety of a domain $\Omega$ in $\mathbb{C}^d$. We may assume
 $$E=\{z\in \Omega| f(z)=0\},$$ for some nonconstant holomorphic function $f$ on $\Omega .$
Let $h$ be a    function in $L_a^2(\Omega\setminus\!E).$  We need only to show
 that $h$ extends holomorphically to $\Omega$ since $E $ has  Lebesgue measure zero.
Let
   $H=(z_1,\cdots,z_{d-1},f)$. Then
   $$JH(z) = \frac{\partial f}{\partial
z_d}(z)$$
for $z$ in $\Omega$.
    For each $\lambda\in E,$ if $ \frac{\partial f}{\partial
z_d}(\lambda)$ does not equal $ 0$,
$JH(\lambda)$ does not equal $0$.
Thus there is a neighborhood $U_1 $ containing
$\lambda$ so that $H$ is
biholomorphic on  a neighborhood of  $\overline{ U_1} .$
So there is a biholomorphic map from $ U_1 \setminus\!E$ onto
 $\mathbb{D}^{d-1}\times (\mathbb{D} \setminus\!\{0\})$.
  Using such the biholomorphic map we can establish an isomorphism between $L_a^2(U_1\setminus\!E)$ and $L_a^2\big(
\mathbb{D}^{d-1}\times (\mathbb{D} \setminus\!\{0\})  \big)$.  Since we have shown above that
each holomorphic function in $L_a^2\big(
\mathbb{D}^{d-1}\times (\mathbb{D} \setminus\!\{0\})  \big)$   extends holomorphically
to a function
  in  $L_a^2(\mathbb{D}^d)$,  each holomorphic function  in
$L_a^2(U_1\setminus\!E)$   extends holomorphically to   a function in  $L_a^2(U_1 ).$
  Thus each holomorphic function $h\in
L_a^2(\Omega\setminus\!E ) $ extends holomorphically to  the subset $ \{z\in E
| \frac{\partial f}{\partial z_d}\neq 0\}\cup (\Omega\setminus\!E)=\{z\in \Omega | \frac{\partial
f}{\partial z_d}\neq 0\}\cup (\Omega\setminus\!E)$ of $\Omega $.
Let
$$\Lambda_j=\{w\in \Omega | \frac{\partial f}{\partial z_j}(w)\neq 0\}, \, 1\leq j \leq d,$$
and   for $r=1,2,\cdots$,
$$\Lambda_{j_1,\cdots,j_r}=\{w\in \Omega | \frac{\partial^r f}{\partial z_{j_1}\cdots \partial z_{j_r}}(w)
\neq 0\}, \, 1\leq j_k \leq d, 1\leq k \leq r. $$
 Repeating the above argument with respect to each variable $z_j$,  we have that  $h$
extends holomorphically to
$$\bigcup_{1\leq j \leq d} \Lambda_j \cup (\Omega\setminus\!E).$$
  Replacing  $f$ by $\frac{\partial^r f}{\partial z_{j_1}\cdots \partial z_{j_r}}$ inductively and repeating the above argument give that  $h$ extends holomorphically to
$$(\Omega\setminus\!E) \bigcup [ \bigcup_{1\leq j \leq d} \Lambda_j  ]\bigcup [ \bigcup_{1\leq j,k \leq d} \Lambda_{j,k}]\bigcup \cdots$$
To finish the proof,  we will show that the union equals $\Omega .$
To do this,  for   a point $w_0 $ in
the complement of $$(\Omega\setminus\!E )\bigcup [ \bigcup_{1\leq j \leq d} \Lambda_j  ]\bigcup [ \bigcup_{1\leq j,k \leq d} \Lambda_{j,k}]\bigcup \cdots ,$$ we have
$$ \frac{\partial^r f}{\partial z_{j_1}\cdots \partial z_{j_r}}(w_0)=0,1\leq  j_1,\cdots,j_r \leq d . $$
This gives that the power expansion of $f$ at $w_0$ equal $0$ and hence  $f$ is identically zero on a neighborhood of
$w_0$. Thus $f\equiv 0$ on $\Omega$. This contradicts that $f$ is nonconstant
  on $\Omega$.
\end{proof}

A  closed set $A$   is said to be  \emph{strongly
$L_a^2$-removable}  in $\Omega$ if
 for each   subdomain $U$ in $\Omega$,
  $U\cap A$ is  $L_a^2$-removable in $U$. Clearly, each closed subset of a
 strongly
$L_a^2$-removable set in $\Omega$ is
$L_a^2$-removable in $\Omega$.  Theorem \ref{removst}  implies that a zero variety
is   strongly $L_a^2$-removable in $\Omega$. The following example gives that the   strongly $L_a^2$-removable
property is strictly stronger than the
  $L_a^2$-removable property.

  \begin{exam} Let $\Omega=\mathbb{D}^2$, and
    $A= \frac{1}{2}\overline{\mathbb{D}}\times \frac{1}{2}\overline{\mathbb{D}}.$
    We will show that each function $f\in L_a^2(\Omega\setminus A)$ extends to a function $\widetilde{f}$ in
 $  L_a^2(\Omega )$. To do this, let $V=(\mathbb{D}\setminus\!\frac{1}{2}\overline{\mathbb{D}})^2.$  {\sl An computation gives that there is a constant $C>0$
 such that }
 $$\int_\Omega |z^\alpha|^2 dV(z)\leq C\int_V |z^\alpha|^2 dV(z), \, \alpha \in \mathbb{Z}_+^2.$$
 By Hartogs' extension theorem, each holomorphic function in $\Omega \setminus A$ extends holomorphically to
  $\Omega.$ Thus each function $f$ in $ L_a^2(\Omega\setminus A)$ extends to a holomorphic function $\widetilde{f}$ on $\Omega$, and
  so we have the Taylor expansion of $\widetilde{f}$:
   $$\widetilde{f}(z)=\sum_{ \alpha \in \mathbb{Z}_+^2} c_\alpha z^\alpha,\, z\in \Omega.$$
   Also we have
   \begin{eqnarray*}
   \int_\Omega |\widetilde{f}(z)|^2 dV(z)
    &=&\sum_{ \alpha \in \mathbb{Z}_+^2}| c_\alpha |^2 \int_\Omega |z^\alpha  |^2 dV(z) \\
 & \leq & \sum_{ \alpha \in \mathbb{Z}_+^2} C| c_\alpha |^2\int_V |z^\alpha  |^2 dV(z) \\
 &=& C \int_V |f(z)|^2 dV(z)<\infty.
 \end{eqnarray*}
This gives that  $\widetilde{f}$ is in $L_a^2(\Omega)$, and hence
  $A$ is $L_a^2$-removable   in $\Omega$.

On the other hand,  $A$ is not strongly $L_a^2$-removable   in $\Omega$. To see this,
let   $$U=\frac{1}{2}\mathbb{D}\times \mathbb{D},$$ and define
$$h(z)= \frac{1}{z_2 },\, (z_1,z_2)\in U\setminus A.$$Since  $h$ is bounded on $U\setminus A$,
we have that $h$ is in $ L_a^2(U\setminus A).$ But it is clear that $h$ can not extend to a holomorphic  function on $U$, and
$h$ does not  extend to a function in $L_a^2(U)$.
Therefore, $A$ is not strongly $L_a^2$-removable   in $\Omega$.
  \end{exam}

We have the following lemma about finite union of strongly $L^2_a$-removable sets.

\begin{lem}\label{union} Suppose $A_1,\cdots,A_n$ are strongly $L_a^2$-removable sets  in $\Omega$ with  Lebesgue measure zero.
Then the union of $A_1,\cdots,A_n$ is  also strongly
$L_a^2$-removable in $\Omega$.\end{lem}
\begin{proof} By induction we need only show that if both $A_1$ and $A_2$ are strongly $L_a^2$-removable   in $\Omega$ and have Lebesgue measure zero,
 then $A_1\cup A_2$ is strongly $L_a^2$-removable   in $\Omega$.
 To see this, let   $U$ be a subdomain in $\Omega$. Note that  $A_2$ is relatively closed in $\Omega$
 and $U\setminus\! A_2=U\cap (\Omega\setminus\! A_2)$ is open and hence $U\setminus\! (A_1\cup A_2)$ is also open.
 Suppose $f$ is a function in $L_a^2(U\setminus\! (A_1\cup A_2) )$.
First we prove that $f$ is in $L_a^2(U\setminus\! A_2) $. Since
$U\setminus\! A_2 $ is open,
 $$U\setminus  A_2= \bigcup_i V_i$$
 where each $V_i$ is a connected component of $U\setminus\! A_2.$
  Since the restriction of $f$ on each $V_i\setminus\!  A_1$ is  in $L_a^2(V_i\setminus\!  A_1) $ and
$A_1 $ is strongly $L_a^2$-removable sets  in $\Omega,$ $f$ extends
analytically to a function in $L_a^2(V_i) $.  Thus $f$ is holomorphic in
$U\setminus\!  A_2$. This implies
$$f \in L_a^2(U\setminus\!  A_2)  $$
as the Lebesgue measure of $A_2$ equals zero.
  Since $A_2$ is strongly $L_a^2$-removable   in $\Omega$,
$f$ extends holomorphically  to a function in $L_a^2(U)$.
 We conclude that $A_1\cup A_2$ is strongly $L_a^2$-removable   in $\Omega$.
\end{proof}

  The following  lemma will be used in the proof of Theorem  \ref{removst2}.
   \begin{lem}\label{inverses} Let  $\mathcal{E} $ be a closed and
   strongly $L_a^2$-removable  set in $\mathbb{C}^d$ with the Lebesgue
   measure zero.
   Suppose  $F: \Omega\to \mathbb{C}^d$ is a holomorphic map on $\overline{\Omega}$ such
that the image of $F$ has an interior point.
 Then
   $  F^{-1}(\mathcal{E})$  is strongly  $L_a^2$-removable in $\Omega $.
\end{lem}
\begin{proof}
Let  $\mathcal{E} $ be a closed and
   strongly $L_a^2$-removable  set in $\mathbb{C}^d$ with the Lebesgue
   measure zero.  Since $F(\Omega )$ has an interior point,  $JF$ does not equal identically zero.
     We will show that $F^{-1}(\mathcal{E})$ has  Lebesgue
   measure zero. To do this, first we observe
$$F^{-1}(\mathcal{E}) \subseteq Z(JF) \bigcup \{z\in \Omega : JF(z)\neq 0 \,\,\  \mathrm{and} \,\,\
 F(z)\in \mathcal{E} \}.$$
  Since $ \{z\in \Omega : JF(z)\neq 0 \,\,\  \mathrm{and} \,\,\
 F(z)\in \mathcal{E} \} $ is contained in a union of countably many biholomorphic images of subsets in $ \mathcal{E}$,
 it must have Lebesgue measure zero. Noting that a zero variety $Z(JF)$ has  the Lebesgue measure zero,
we have that $F^{-1}(\mathcal{E})$ has Lebesgue measure zero.

Let $V$ be an open subdomain in $\Omega$. We will show
that   $F^{-1}(\mathcal{E}) \cap  V$ is  $L_a^2$-removable  in $V$.
 To do this,  we note that
for each $\lambda \in \Omega,$ if $JF(\lambda)\neq 0,$ then there is
an open neighborhood $U(\lambda)$ of $ \lambda $ such that $F$ is
injective on   $\overline{U(\lambda)}$.
 Since
 $$F^{-1}(\mathcal{E})\cap U(\lambda)=\{z\in U(\lambda)| F(z)\in \mathcal{E}\}=
 (F|_{U(\lambda) })^{-1}( \mathcal{E})  ,$$
and $\mathcal{E}$ is strongly $L_a^2$-removable in ${\mathbb C}^d$,  $F^{-1}(
\mathcal{E})\cap U(\lambda) $ is $L_a^2$-removable in $ U(\lambda)$.
Thus for each
$h\in L_a^2(V\setminus\!F^{-1}(\mathcal{E}))$,
 $h$ extends holomorphically to
 $$\cup_{\lambda \in \{z:JF(z)\neq 0\}}U(\lambda ).$$
 Since
 $$V\setminus\! Z(JF)=V \cap \{z:JF(z)\neq 0\}\subset \cup_{\lambda \in \{z:JF(z)\neq 0\}}U(\lambda ),$$
 $h$ extends holomorphically to  $V\setminus\! Z(JF).$  Since
  we have shown above that \linebreak $ \{z\in \Omega : JF(z)\neq 0 \,\,\  \mathrm{and} \,\,\
 F(z)\in \mathcal{E} \} $ has the Lebesgue
   measure zero,  $h$ is in \mbox{$L_a^2(V\setminus\! Z(JF))$}.   As $Z(JF)$ is a zero variety,
   by Theorem \ref{removst}, $h$ extends holomorphically to $V$. Hence we obtain that   $F^{-1}(\mathcal{E }) \cap V$ is   $L_a^2$-removable in $V$.
This means that  $F^{-1}(\mathcal{E} )$  is strongly
$L_a^2$-removable  in $\Omega $.
\end{proof}
To prove Theorem  \ref{removst2}, we need some lemmas from several complex analysis.
  \begin{lem}\label{dimrem}If $E$ is
    a relatively closed subset of a domain $D$ in $\mathbb{C}^d$ and   of zero Hausdorff measure $h_{2d-2}(E) $,
    then each function holomorphic on $D \backslash E$ has a holomorphic continuation  to $D$.\end{lem}
  Lemma \ref{dimrem} is Proposition 3 in  \cite[p.298]{Ch}. We need a special case of the proposition on \cite[p. 41]{Ch}.

\begin{lem}\label{image2} Let $A$ be an analytic subvariety of a domain in $\mathbb{C}^d,$ and
$F: A\to \mathbb{C}^n$  be a holomorphic map. Then $F(A)$ is contained in a union of
at most countable analytic subvarieties of domains in $\mathbb{C}^n$, of
dimensions not larger than the dimension $\dim F $ of $F$. Consequently,   $\dim F(A)\leq \dim A.$
\end{lem}

\begin{lem} \label{image1} Let
$F:  \overline{\mathbb{D}}^{d-1}\to \mathbb{C}^d$
  be  a holomorphic map.
    Then $F(\overline{ \mathbb{D}}^{d-1})$   is  strongly $L_a^2$-removable in $ \mathbb{C}^d.$
\end{lem}

\begin{proof} Let $r$ denote  the maximal  rank of Jacobian matrix of $F$ on
 $ \overline{\mathbb{D}}^{d-1}$. If $r\leq d-2,$ then the dimension $\dim F$ of $F$ is no larger than $d-2$
  \cite[p. 40]{Ch}.
 Then by Lemma \ref{image2},  $F(\overline{ \mathbb{D}}^{d-1})$  is contained
 in a union of at most countable analytic subvarieties $E_j$ of domains, which are of dimensions $\leq d-2,$ and hence
 their  Hausdorff measures $h_{2d-2}(E_j)$ are zero. Therefore $h_{2d-2}F(\overline{ \mathbb{D}}^{d-1})=0$.
 Since the Lebesgue measure of $F(\overline{ \mathbb{D}}^{d-1})$ equals zero, by Lemma \ref{dimrem},
 we immediately have that  $F(\overline{ \mathbb{D}}^{d-1})$   is  strongly $L_a^2$-removable in $ \mathbb{C}^d.$

\vskip2mm
Now assume  that  the maximal  rank of Jacobian matrix of $F$ on
 $ \overline{\mathbb{D}}^{d-1}$
 is $d-1$. Letting $F=(f_1,\cdots, f_d)$, we may assume that the holomorphic function
 $$ \varphi (z_1,\cdots,z_{d-1})= \det \frac{\partial (f_1,\cdots,f_{d-1})}{\partial (z_1,\cdots, z_{d-1})}$$
  is not identically zero on $ \overline{\mathbb{D}}^{d-1}$. Let
   $$  \overline{\mathcal{Z}}=\{z\in \overline{ \mathbb{D}}^{d-1}: \varphi(z)=0\}.$$
 Let $\Omega $ be any domain  in $\mathbb{C}^d.$ For each $\varepsilon>0,$
let $V(\varepsilon)$ denote the $\varepsilon$-neighborhood of $F(\overline{\mathcal{Z}})$,
$$V(\varepsilon)=\{w\in \mathbb{C}^{d}: \inf_{\lambda\in F(\overline{\mathcal{Z}})}|w-\lambda|<\varepsilon\}$$
 and let   $$\Omega_\varepsilon =\Omega\setminus \overline{V(\varepsilon)},$$
As $F$ is a continuous map,  there is a positive  integer $k$ such that the $\frac{1}{k}$-neighborhood $U(k)$ of $\overline{\mathcal{Z}}$
satisfies $$F(U(k)) \subseteq V(\varepsilon).$$
Thus we have \begin{equation}\label{eqi} F(\overline{\mathbb{D}}^{d-1}) \setminus  V(\varepsilon)
=F\big(\overline{\mathbb{D}}^{d-1} \setminus U(k)\big) \setminus  V(\varepsilon).\end{equation}
Let $$F_0=(f_1,\cdots,f_{d-1}) .$$ For each point $w\in \overline{\mathbb{D}}^{d-1} \setminus U(k)$,
the Jacobian of $F_0$ at $w$ does not vanish, and then there is a ball $U_w$
centered at $w$ such that $F_0$ is univalent  on a neighborhood of $\overline{U_w}$.
 Since $F_0: U_w\to F_0(U_w)$ is a biholomorphic map, for each $(\lambda_1, \cdots , \lambda_{d-1}, \lambda_d)\in F(U_w)$ we get
 $$\lambda_1=f_1(z_1, \cdots , z_{d-1}), \lambda_2=f_2(z_1, \cdots , z_{d-1}), \cdots , \lambda_{d-1}=f_{d-1}(z_1, \cdots , z_{d-1}),$$
 $$\lambda_d=f_d(z_1, \cdots , z_{d-1})=f_d(F_0^{-1}(\lambda_1 , \lambda_2, \cdots , \lambda_{d-1})).$$
 Thus $F(U_w)$ is a zero variety.
 Since   $ \overline{\mathbb{D}}^{d-1} \setminus U(k)$
 is compact, \linebreak $ \overline{\mathbb{D}}^{d-1} \setminus U(k)$
 is contained in a finite union of open balls   $\{U_w\}.$ Then
 $F\big(\overline{\mathbb{D}}^{d-1} \setminus U(k)\big)$ is contained in a union of finitely many
 zero varieties $F(U_w)$, and by  Theorem \ref{removst} and Lemma \ref{union},
  $F\big(\overline{\mathbb{D}}^{d-1} \setminus U(k)\big)$ is
  strongly $L_a^2$-removable in $\Omega_\varepsilon$. By (\ref{eqi}), $F(\overline{\mathbb{D}}^{d-1}) \setminus  V(\varepsilon)$
   is strongly $L_a^2$-removable in $\Omega_\varepsilon$. Thus $F(\overline{\mathbb{D}}^{d-1})$
   is strongly $L_a^2$-removable in $\Omega_\varepsilon$.
\vskip1mm
   For each holomorphic function $h$ in 
 $L_a^2(\Omega \setminus F(\overline{\mathbb{D}}^{d-1}) )$, $h$ is in
 $L_a^2(\Omega_\varepsilon \setminus F(\overline{\mathbb{D}}^{d-1}) )$.
  Since $F(\overline{\mathbb{D}}^{d-1})$
   is strongly $L_a^2$-removable in $\Omega_\varepsilon$,  $h $ is in $L_a^2(  \Omega_\varepsilon)$,
    and  by arbitrariness of $\varepsilon$,
    $h$ is holomorphic on  $\Omega \setminus F(\overline{\mathcal{Z}})$.
   Noting that  $\dim \overline{\mathcal{Z}} \leq d-2,$  by Lemma \ref{image2} we have $$\dim F(\overline{\mathcal{Z}})\leq d-2.$$ Then
   the Hausdorff measure  $h_{2d-2}(F(\overline{\mathcal{Z}}))=0,$ and by Lemma \ref{dimrem}
       $h$ extends holomorphically to $\Omega .$ Noting that $h\in
  L_a^2(\Omega \setminus F(\overline{\mathbb{D}}^{d-1}) )$ and the Lebesgue measure of $F(\overline{\mathbb{D}}^{d-1}) $ equals zero, we have $h\in L_a^2(\Omega ).$ Thus $F(\overline{\mathbb{D}}^{d-1}) $ is strongly $L_a^2$-removable in $\Omega $, and so it is
   strongly $L_a^2$-removable in $\mathbb{C}^d$
   to complete the proof.
\end{proof}

Now we are ready to present the proof of Theorem   \ref{removst2}.
\vskip2mm
\noindent \textbf{Proof of Theorem  \ref{removst2}.}
 Let   $E $ be a
good {local}  zero variety of a domain $\Omega$. We will investigate the structure of $\overline{E}$.
To do this, let ${\mathcal R}_E$ be the set of these
  points $\lambda$ in $\overline{E}  $ with the following property:

 \emph{ There is an open neighborhood $W$ of $\lambda$ and a holomorphic function $h$ on $\overline{W}$ such that
$$\overline{E}\cap W= \overline{\Omega}\cap  \{z\in W: h(z)=0\},$$
and $ \{z\in W: h(z)=0\} $ is biholomorphic to  an open set in  $\mathbb{C}^{d-1}$.}
\vskip1mm
Then { ${\mathcal R}_E$ is a relatively open set   in $\overline{E}$.}
Let $ \mathcal{S}_E$  denote   the complement  of ${\mathcal R}_E$ in $\overline{E}$. Hence  $ \mathcal{S}_E$ is a closed set.

  We will show that
  the  Hausdorff dimension of   $\mathcal{S}_E$ is at most $2d-4$.
To do so, let $\lambda$ be a point in  $\overline{E}$. Since $E$ is a good local zero variety, there are a holomorphic function $h$ and an open neighborhood $W$ of $\lambda$ such that
$$W\cap \overline{E}=\{z\in W\cap \overline{\Omega}: h(z)=0\}.$$
Let $E_W=W\cap \overline{E}$, $(\mathcal{R}_E)_W=\mathcal{R}_E\cap W,$ and
$$(\mathcal{S}_E)_W=\mathcal{S}_E\cap W .$$
 {The proposition in} \cite[p. 22]{Ch}  gives that ${\mathcal R}_E$ is a
 complex manifold with dimension $d-1$ and the { proposition in}
 \cite[p. 20]{Ch} implies that $(\mathcal{S}_E)_W$ is the singular locus of
  $E_W$ which is made of the singular points of $E_W$ and
   finitely many complex submanifolds with dimensions less than $d-1$.
    By Corollary 3 \cite[p. 24]{Ch} we have that the Hausdorff dimension of $(\mathcal{S}_E)_W$ is at most $2d-4$.

Since $\mathcal{S}_E$ is compact, $\mathcal{S}_E$ is covered by   finitely many open sets $\{W_i\}_{i=1}^n$ and the Hausdorff dimension of each of $\{(\mathcal{S}_{E})_{W_i}\}_{i=1}^n$ is at most $2d-4$.
Noting
$$\mathcal{S}_E= \cup_{i=1}^{m}(\mathcal{S}_{E})_{W_i},$$
we get that the Hausdorff dimension of
$\mathcal{S}_E $ is at most $2d-4$.

\vskip2mm
The rest of the proof is similar to the proof of Lemma \ref{image1}.
For each $\varepsilon>0,$ let $V(\varepsilon)$ be the
$\varepsilon$-neighborhood of $F(\mathcal{S}_E)$. Since $F$ is holomorphic on $\overline{\Omega}$,  there is a positive integer $k$ such
that for each point $w$ in the $\frac{1}{k}$-neighborhood $\mathcal{N}(k)$ of $\mathcal{S}_E$,
 \begin{equation}F(w)\in V(\varepsilon). \label{eqa}\end{equation}
  This gives that for each point $\lambda$ in $\overline{E}\setminus \mathcal{N}(k)$, $\lambda$ is   in $ \mathcal{R}_E$.
 Thus there is a neighborhood $U_\lambda$ of $\lambda$ such that $ U_\lambda \cap \overline{E} $ is contained in the
image of $\overline{\mathbb{D}}^{d-1} $ under a biholomorphic map,
  and so $F(U_\lambda \cap \overline{E})$ is also contained in the image of
  $\overline{\mathbb{D}}^{d-1} $ under a holomorphic map. By Lemma \ref{image1}, $F(U_\lambda \cap \overline{E})$
  is  strongly $L_a^2$-removable in $ \mathbb{C}^d.$
Since $\overline{E}\setminus \mathcal{N}(k)$ is compact, we have that $\overline{E}\setminus \mathcal{N}(k)$
 is contained in a union of finitely many  sets $ U_\lambda \cap \overline{E} $.
 Thus
  Lemma \ref{union} gives that
   $F(\overline{E}\setminus \mathcal{N}(k))$ is
  strongly $L_a^2$-removable in $\mathbb{C}^d$.
   By (\ref{eqa}),  $$F(\mathcal{N}(k))\subseteq V(\varepsilon),$$
   we have
$$F( \overline{E}) \setminus  V(\varepsilon) \subseteq F(\overline{E}\setminus \mathcal{N}(k)),$$
to obtain that $F( \overline{E}) \setminus  V(\varepsilon)$ is strongly $L_a^2$-removable in $\mathbb{C}^d$.

 For any domain $U$ in $\mathbb{C}^d$ and a function $h$ in $L_a^2(U\setminus F(\overline{E}))$,
  $h$ belongs to \linebreak $   L_a^2(\big(U \setminus  \overline{V(\varepsilon)}\big)\setminus F(\overline{E})).$  Since $F( \overline{E}) \setminus  V(\varepsilon)$ is
strongly $L_a^2$-removable, we have that
  $h$ is in $ L_a^2(U \setminus  \overline{V(\varepsilon)})$ to get that
   $h$ is holomorphic in $U \setminus F(\mathcal{S}_E)$.
 {Since $F$ is holomorphic on a neighborhood of $\overline{E}$,
$F$ satisfies the Lipschitz condition on some neighborhood of the subset $\mathcal{S}_E $ of $\overline{E}$.
  Thus \cite[p. 347, Property 4]{Ch} gives that
 the Hausdorff dimension of $ F(\mathcal{S}_E)$ is not
 larger than that of  $ \mathcal{S}_E$, which is  at most $ 2d-4 .$}
 This gives 
     that the Hausdorff measure  $$h_{2d-2}(F(\mathcal{S}_E))=0.$$
Thus by Lemma \ref{dimrem},  we have that $h$ extends holomorphically to $U$
  and hence $h$ is in $ L_a^2(U),$  to get that  $F( \overline{E}) $ is
   strongly $L_a^2$-removable in $ \mathbb{C}^d$.

\vskip2mm

  Noting that $F(\overline{E})$ has the   Lebesuge measure $0$,
 by Lemma \ref{inverses}  we conclude that $F^{-1}(F(E) )$
 is    strongly $L_a^2$-removable $\Omega$ and hence   $L_a^2$-removable in $\Omega$ to complete
 the proof of Theorem  \ref{removst2}.
$\hfill \square$
\vskip2mm
 \begin{rem} The  last two paragraphs above indeed  show that  if  $E $ is  a good local zero variety  of a domain $\Omega$ in  $\mathbb{C}^d,$    and if
 $F: \Omega\to \mathbb{C}^d$ is holomorphic on $\overline{\Omega}$ such
that the image of $F$ contains an interior point, then
  $F( \overline{E}) $ is    strongly $L_a^2$-removable in $ \mathbb{C}^d$, and
 $F^{-1}(F(E) )$
 is    strongly $L_a^2$-removable in $\Omega$.



 A simple version of  Remmert's  Proper Mapping theorem
 (\cite[p. 65]{Ch} or \cite{Re1,Re2})
  says
that if $f:\Omega_1\to \Omega_2$ is a holomorphic proper map  and $\mathcal{Z}$ is a
subvariety of $\Omega_1$, then $f(\mathcal{Z})$ is a subvariety of $ \Omega_2.$
 Since a subvariety can be locally represented as the intersection of finitely many (locally-defined) zero varieties,
    $f(\mathcal{Z})$ is $L_a^2$-removable in $\Omega_2$.

    \end{rem}

     We also have a corollary of Theorem \ref{removst}.
\begin{cor} \label{removst3} Suppose that  $E $ is a zero variety of a domain $\Omega$ in $\mathbb{C}^d$, and
$F $ is a holomorphic proper map  on $ \Omega $. Then    $
F^{-1}(F(E))$  is  relatively closed and $L_a^2$-removable in $ \Omega $.
\end{cor}
\begin{proof}
Suppose that  $E $ is a zero variety of a domain $\Omega$, and
$F $ is a holomorphic proper map  on $ \Omega $.
 Note that by  Remmert's Proper Mapping theorem  $F(E)$ is a subvariety of $F(\Omega)$. { Thus $F(E)$ is relatively closed and is a {local} zero variety,
and so is   $  F^{-1}(F(E))$. By  Theorem \ref{removst}, $
F^{-1}(F(E))$ is strongly
 $L_a^2$-removable in $\Omega$.  } Thus  each function in
 $L_a^2(\Omega\setminus\! F^{-1}(F(E)))$ extends holomorphically to $\Omega.$
  Noting that the Lebesgue measure of $  F^{-1}(F(E))$ equals zero,  we conclude that
  $  F^{-1}(F(E))$ is  $L_a^2$-removable to complete the proof. \end{proof}

It is known that if  $E $ is a zero variety of a domain $\Omega$,
then $\Omega\setminus\!E$ is connected  \cite[Chapter 14]{Ru2}. We need the following result.

\begin{prop} \label{conn} Suppose that  $E $ is a {local} zero variety of a domain $\Omega$ in $\mathbb{C}^d$. Then we have the following:
\begin{itemize}
\item[$(i)$] if $E$ is good and $F: \Omega\to \mathbb{C}^d$ is holomorphic on $\overline{\Omega}$ such
that the image of $F$ has an interior point, then  $\Omega \setminus\! F^{-1}(F(\overline{E}))$ is connected.
\item [$(ii)$]if  $F$ is a holomorphic proper map on $\Omega,$ then  $\Omega \setminus\! F^{-1}(F(E))$ is connected.
\end{itemize}
\end{prop}
 \begin{proof} To prove (i) we suppose that $E $ is a good {local}  zero variety of a domain $\Omega$, and
 $F: \Omega\to \mathbb{C}^d$ is a  holomorphic map such
that the image of $F$ has an interior point. First we note that
$$F^{-1}(F(\overline{E})) \subseteq Z(JF)\bigcup\{z\in \Omega: JF(z)\neq 0, F(z)\in F(\overline{E})\}.$$

If $JF(\lambda)\neq 0,$   there is an open neighborhood
$U(\lambda)$ of $ \lambda $ such that $F$ is biholomorphic on
$U(\lambda)$. This gives that $ \{z\in \Omega: JF(z)\neq 0, F(z)\in
F(\overline{E})\} $ is contained in  a union of countably many sets $(F|_{U(\lambda) })^{-1}\big(F(\overline{E})\big).$
The proof of Theorem \ref{removst2} gives that $F(\overline{E})$  is contained in a union of finitely
many sets whose Hausdorf dimensions are at most $2d-2$. Hence the Hausdorff measure $h_{2d-2}(F(\overline{E}))$ is finite. Thus $ \{z\in \Omega: JF(z)\neq 0, F(z)\in
F(\overline{E})\} $ is contained in  a union of countably many sets whose Hausdofff dimensions are at most $2d-2$.
 As
\cite[Theorem 14.4.9]{Ru2} states that a zero
variety of a domain in  $\mathbb{C}^d$ can be represented as a union
of countably compact sets  $L_n$ whose Hausdorff measures
$h_{2d-2}(L_n)<\infty,$    we obtain that the Hausdorff dimension of $F^{-1}(F(\overline{E}))$ is at most $2d-2$.
    \cite[Theorem 14.4.5]{Ru2} states that for a connected domain $\mathcal{U}$ in $\mathbb{R}^{2d}$, if
   a relatively closed set  $G$ can be written as the union of countably many
   compact sets $K_n$ with the Hausdorff measure $h_t(K_n) <\infty$ for some $ t\in (0, 2d-1),$
   then $\mathcal{U}\setminus\! G$ is connected. Then we conclude that $\Omega\setminus\! F^{-1}(F(\overline{E}))$  is connected.

   To prove (ii),  by Corollary \ref{removst3} and its proof, we have that if  $F$ is a holomorphic proper map on $\Omega,$
   then $ F^{-1}(F(E))$ is relatively closed, and is  a {local} zero variety.
   Then by the same argument above, we conclude that
$\Omega \setminus\! F^{-1}(F(E))$ is connected.
\end{proof}

Proposition \ref{conn} (i) may fail if the image of $F$ does not have any  interior points.
For example, let  $\Omega_0=\{z\in\mathbb{D}: |Re \, z|
<\frac{1}{2}\}$,
 $E=\{z \in \Omega_0\times \mathbb{D}: z_1=z_2\}$ and $$F(z_1,z_2)=(z_2,z_2).$$ Letting $\Omega=\Omega_0\times \mathbb{D},$
 we have  that $\Omega \setminus\!F^{-1}(F(\overline{E})) =   \Omega_0\times ( \mathbb{D} \setminus\!\Omega_0 )$
 is not connected.

\section{\label{sect4} Proof of Theorem \ref{thm11}}
 In this section we will present the proof of   Theorem   \ref{thm11}.
  We also need the  notion of representing local inverse, which is of great importance in the proof of Theorem \ref{thm11}. If for some operator $T$ in $\mathcal{V}^*(\Phi ,\Omega)$, $T$ has the following representation:
\begin{equation*}
  Th(w) = \sum_{j=1}^N c_jh(\rho_j(w))(J\rho_j)(w) ,\, h\in L_a^2(\Omega), \, w \in \Delta,
  \end{equation*} on an open domain $\Delta$ where $c_j $ are constants and
  all $\rho_j$ are local inverses of $\Phi$ on $\Delta$, then
$\rho_k$ is called \emph{a representing local inverse}
 \emph{for} $\mathcal{V}^*(\Phi ,\Omega)$ on $\Delta$ provided that $c_k$ does not equal $0$.
The representing local inverse is first introduced in \cite{T1}
on the unit disk $\mathbb{D}$.

From now on,   let  $\Phi$   denote $(\phi _1,\cdots,\phi _d)$, where each $\phi
_j(1\leq j\leq d)$ is a holomorphic function on $\Omega$. Let $Z$
denote the set of zeros of the determinant of the Jacobian  $J\Phi $ of
$\Phi $.
\vskip2mm
    We are ready to present the proof of Theorem \ref{thm11}.
\vskip2mm  \noindent \textbf{Proof of Theorem \ref{thm11}.}
  Suppose  that the interior of $\overline{\Omega}$ equals
   $\Omega$ and
   $\Phi  :\Omega\to \C^d$ is  holomorphic  on a neighborhood of $\overline{\Omega}$  such that
 the image of  $\Phi  $ has an interior point. First we show that
  $\mathcal{V}^*(\Phi ,\Omega)$ is of finite dimension.
  Since the image of  $\Phi  $ has an interior point,
 the complex dimension of $\Phi (\Omega)$ equals
$d$. This gives that  $J \Phi $ does not identically vanish.  Then
$\Phi (\overline{Z})$ is a closed set with the Lebesgue measure zero.
As $\Phi$ may map some points in $\Omega$ to some points in $\partial \Omega$, for
each $\lambda \in  \Omega\setminus\!\Phi ^{-1}(\Phi  (\overline{Z}))$,
  $\Phi^{-1}(\Phi(\lambda))\cap \partial\Omega $ may not be empty.
Write
  $\Phi^{-1}(\Phi(\lambda))\cap \overline{\Omega}=\{\lambda_1,\cdots,\lambda_N\}$. Some of $\lambda_j$ may be in $\partial \Omega$.
We will use the property of $\Omega$ that
the interior of $\overline{\Omega}$ equals $\Omega$ to show that for some neighborhood of $\lambda$, the images of an open subset of  the neighborhood under local inverses are contained either in $\Omega$  or in $C^d\setminus \overline{\Omega }.$

 For each point $\lambda_j$, since the Jacobian $J\Phi$ does not
  vanish at $\lambda_j$, there exists a neighborhood of $\lambda_j$
  which is mapped biholomorphically to a neighborhood of $\Phi(\lambda_j)=\Phi(\lambda ).$
Therefore there exist  a disk  $\Delta$ containing $\lambda$,   $N$ domains
 $\Delta_1,\cdots,\Delta_N$ and biholomorphic maps $\rho_1,\cdots, \rho_N$ such that
\begin{equation}\label{213} \bigsqcup_{j=1}^N \Delta_j =\Phi ^{-1}(\Phi (\Delta)) ,\end{equation}
  $$\rho_j(\Delta_1)=\Delta_j  \quad \mathrm{and} \quad \Phi \circ
\rho_j=\Phi ,\, 1\leq j \leq N, $$ where $\rho_1(z)\equiv z$,
$\Delta_1=\Delta.$ 
Note  some $\Delta_i$ may have nonempty intersection with $\partial \Omega.$
  In fact, $\Phi^{-1}(\Phi(\lambda))\cap \partial \Omega $ may not be empty. If for some $i$,  $\rho_i(\lambda)$ is in $\partial \Omega $,
  then $\rho_i(\Delta_1)\cap \partial \Omega\neq \emptyset$.
  We will obtain some formula  like
 $$
  Sg(w) = \sum_{j }  c_jg(\rho_j(w))(J\rho_j)(w) ,\, w \in \Delta ,
$$  which   $c_i=0$ if  $\rho_i(\lambda)\in  \partial \Omega $ and $g\in L_a^2(\Omega)$. Here the disk $\Delta$ may shrink. 

   We will  derive the above formula and show $c_2=0$ if $\rho_2(\lambda)\in  \partial \Omega$.
Since  the interior of  $\overline{\Omega}$ equals $\Omega$, by Remark 1.4(2)
 there is a domain $U$ contained in $\Delta_2 (\Delta_2 =\rho_2(\Delta)) $ such that $U $ has no intersection with $\Omega.$
 Set $ \widetilde{\Delta} =\rho_2^{-1}(U) \subseteq \Delta,$ and we have $\lambda\not\in \widetilde{\Delta}$.
By shrinking $\widetilde{\Delta}$ finitely many times (we still keep the notation $\widetilde{\Delta}$),  one has that
for each $j$ either $\rho_j(\widetilde{\Delta}) \subseteq \Omega $ or
$\rho_j(\widetilde{\Delta}) \cap \overline{\Omega} =\emptyset. $ By (\ref{213}) we immediately get
$$\bigsqcup_{j=1}^N \widetilde{\Delta_j} =\Phi ^{-1}(\Phi (\widetilde{\Delta}))$$
Let $\Lambda= \{j:1\leq j\leq N, \rho_j(\widetilde{\Delta}) \subseteq \Omega \}$. Thus
\begin{equation}  \bigsqcup_{j\in \Lambda}  \widetilde{\Delta_j} =\Phi ^{-1}(\Phi (\widetilde{\Delta}))\cap \Omega .\label{212}\end{equation}

We will give the representation of those operators in the von
Neumann algebra $\mathcal{V}^*(\Phi ,\Omega)$  and
 the idea comes from \cite{GH1}. To do this, let $S$ be a unitary operator in $\mathcal{V}^*(\Phi ,\Omega)$. Then
  $S$ commutes with every operators in $\{ M_{\phi_j}  :1\leq j \leq d \}$. For any functions $g$ and $h$ in $L_a^2(\Omega) $, let
$$\widetilde{g}=Sg \quad  \mathrm{and} \quad  \widetilde{h}=Sh.$$
Thus  for  any  polynomial $P$ of $z_1, \cdots, z_d$, letting
$$P(\Phi(z))=P(\phi_1(z), \cdots \phi_d(z))$$
we have
$$M_{P(\Phi )}S=SM_{P(\Phi )},$$
so for  any two polynomials $P$ and $Q$ of $z_1, \cdots, z_d$,
\begin{eqnarray*}\langle P(\Phi )\widetilde{g},Q(\Phi )\widetilde{h}\rangle
&=&\langle M_{P(\Phi)} Sg, M_{Q(\Phi)}Sh\rangle   \\
&=& \langle SM_{P(\Phi)}g, SM_{Q(\Phi)}h\rangle  \\
&=&\langle P(\Phi )g,Q(\Phi )h \rangle .
\end{eqnarray*}
 This implies
\begin{equation}
\int_{\Omega}\Big( (P\overline{Q})\circ \Phi (w)g(w)\overline{ h}(w)
-(P\overline{Q})\circ \Phi (w)
\widetilde{g(w)}\overline{\widetilde{h(w)}}\Big) dV(w)=0.
\label{2.1}
\end{equation}
 Now let
$$\mathcal{X}=span\,\{p\overline{q}\,:\ p,q \ \mathrm{are \ polynomials\ in}\  d \ \mathrm{ variables} \}.$$
By the Stone-Weierstrass Theorem,  each continuous function on
$\overline{\Phi (\Omega)}$ $ $ can be uniformly approximated by
members in $\mathcal{X} $. Thus   $(\ref{2.1})$ gives
\begin{equation}
\int_{\Omega}\Big(  u(\Phi (w))g(w)\overline{ h}(w) -u(\Phi (w))
\widetilde{g(w)}\overline{\widetilde{h(w)}} \Big) dV(w)=0, u\in
C(\overline{\Phi (\Omega)}) .      \label{2.2}
\end{equation}
 By Lebesgue's  Dominated Convergence Theorem,
 (\ref{2.2})  holds for all  $u$ in $ L^\infty(\Phi (\Omega)).$ Thus
for each $u$ in $ L^\infty(\Phi (\widetilde{\Delta})),$  (\ref{2.2})  gives that
$$\int_{\Omega} \chi_{\Phi (\widetilde{\Delta})}(\Phi (w))  u(\Phi (w))g(w)\overline{ h}(w) dV(w)
 =\int_{\Omega} \chi_{\Phi (\widetilde{\Delta})}(\Phi (w)) u(\Phi (w))\widetilde{g(w)}\overline{\widetilde{h(w)} } dV(w) , $$
 and hence
$$\int_{\Phi ^{-1}(\Phi (\widetilde{\Delta}) )\cap \Omega}  u(\Phi (w))g(w)\overline{ h}(w) dV(w)
 =\int_{\Phi ^{-1}(\Phi (\widetilde{\Delta}))\cap \Omega} u(\Phi (w))\widetilde{g(w)}\overline{\widetilde{h(w)} } dV(w) . $$
 By (\ref{212}) we get $$\int_{\bigsqcup_{j\in \Lambda}  \widetilde{\Delta_j}  }  u(\Phi (w))g(w)\overline{ h}(w) dV(w)
 =\int_{\bigsqcup_{j\in \Lambda}  \widetilde{\Delta_j} } u(\Phi (w))g(w) \overline{\widetilde{h(w)} } dV(w) . $$
 and then 
  making changes of variables gives
$$\int_{\widetilde{\Delta}}  u(\Phi (z)) \sum_{j\in \Lambda}  (g \overline{h})\circ \rho_j(z)|(J\rho_j)(z)|^2dV(z)
=\int_{\widetilde{\Delta}}u(\Phi (z))\sum_{j\in \Lambda}
(\widetilde{g}\overline{\widetilde{h}})\circ
 \rho_j (z)|(J\rho_j)(z)|^2 dV(z).$$
  Noting that $\Phi $ is injective on $\widetilde{\Delta}$ and $u$ can be an arbitrary function in $L^\infty(\Phi (\Delta))$,
  we immediately have
 \begin{equation}
  \sum_{j\in \Lambda}  (g \overline{h})\circ \rho_j(z)|(J\rho_j)(z)|^2
=  \sum_{j\in \Lambda}  (\widetilde{g}\overline{\widetilde{h}})\circ
 \rho_j (z)|(J\rho_j)(z)|^2 ,\    \
z \in \Delta. \label{2.3}
\end{equation}
In fact, first  (\ref{2.3}) holds almost everywhere on $\Delta$,
and then by continuity (\ref{2.3}) holds on $\Delta$.
  Let $\mathcal{H}$ be the Bergman space
over $\Delta$.  Let
$$e_{g}^j=g(\rho_j(z))(J\rho_j)(z) \quad \mathrm{and} \quad f_g^j=\widetilde{g}(\rho_j(z))(J\rho_j)(z),
1\leq j\leq N, \, g \in L_a^2(\Omega) .$$ These functions $e_{g}^j$ and $f_g^j$
are in $\mathcal{H}.$ By $(\ref{2.3})$,
$\sum_{k\in \Lambda}  e_g^k \otimes e_h^k$ and $\sum_{k\in \Lambda}  f_g^k \otimes
f_h^k$ have the same Berezin transform. Since the Berezin transform is injective, we have
$$\sum_{k\in \Lambda}  e_g^k \otimes e_h^k=\sum_{k\in \Lambda}  f_g^k \otimes f_h^k
,\  g,h\in  L_a^2(\Omega).$$ Let $K=\sharp \Lambda $.   By  Lemma 3.5 in  \cite{GH1}
  there is an $K\times K$ unitary numerical matrix $W$ such that
  $$W \Big(g(\rho_k(w))J\rho_k (w)\Big)_{k\in \Lambda}=
\Big( \widetilde{g}(\rho_k(w))J\rho_k (w)\Big)_{k\in \Lambda},\,
  w  \in  \Delta,$$
Note that $1\in \Lambda$.
By expanding the first raw of $W$, we have that there exist $N$ constants
$c_1,\cdots,c_N$ such that
$$\widetilde{g}(\rho_1(w))J\rho_1 (w) =\sum_{ j\in \Lambda} c_jg(\rho_j(w))(J\rho_j)(w)= \sum_{j=1}^N c_jg(\rho_j(w))(J\rho_j)(w) ,$$
where $c_j=0$ if $j\not\in \Lambda$,
thus to get
\begin{equation*}
  Sg(w) = \sum_{j=1}^N c_jg(\rho_j(w))(J\rho_j)(w) ,\, w \in \widetilde{\Delta} , g\in L_a^2(\Omega)
\end{equation*}
as $\rho_1(w)\equiv w.$  In particular, $c_2=0.$
Thus we get
 $$
  Sg(w) = \sum_{j\in \Lambda }  c_jg(\rho_j(w))(J\rho_j)(w) ,\, w \in \widetilde{\Delta} ,
$$
where $g$ ranges over analytic polynomials.
Since each $\rho_j$ is holomorphic on $\Delta (\Delta \subseteq \Omega),$ both sides of the above equality are holomorphic on $\Delta$. Thus  the above equality also holds
for each $w\in \Delta$.
As $\rho_2(\lambda ) \in \partial \Omega,$ we have that   $c_2=0$. Similarly we conclude that
$c_i=0$ if   $\rho_i(\lambda ) \in \partial \Omega $. That is,
for each $j  $ such that $c_j\neq 0,$ $\rho_j(\lambda ) \in \Omega$. Then
 \begin{equation*}
  Sg(w) = \sum_{j\in \Lambda}  c_jg(\rho_j(w))(J\rho_j)(w) ,\, w \in \Delta ,
  \end{equation*} holds for each $g\in L_a^2(\Omega).$

  In summary,    we write
  \begin{equation} \label{2.4}
  Sg(w) = \sum_{j=1}^N c_jg(\rho_j(w))(J\rho_j)(w) ,\, w \in \Delta , g\in L_a^2(\Omega)
\end{equation}
where it is assumed the symbols $j$ are   rearranged such that for such $j$ satisfying $c_j\neq 0$, one has $\rho_j(w)\in \Omega$ for each $w\in \Delta.$

 Since each operator in a von Neumann algebra is a linear combination of
four unitary operators in the von Neumann algebra  (\cite[Proposition 13.3]{Con} , \cite[Theorem 10.6]{Zhub}) each operator  $S$   in
$\mathcal{V}^*(\Phi ,\Omega)$ has the same form as  (\ref{2.4}). Thus $S$ is completely determined by $(c_1,\cdots,c_N)$ in the formula (\ref{2.4}) on $\Delta.$  Noting
that    vectors $(c_1,\cdots,c_N)$  belong to a subspace of
$\mathbb{C}^N$, we have
$$\dim \mathcal{V}^*(\Phi ,\Omega)\leq N.$$

Next we will find   finitely many generators of
$\mathcal{V}^*(\Phi ,\Omega)$. To do so, recall  that a local inverse $\rho$ of $\Phi :\Omega\to \C^d$ is called  \emph{admissible}
  if for each curve $\gamma$ in $\Omega\setminus\! \overline{\Phi^{-1} (\Phi(Z)) } $, $\rho$ admits analytic
continuation with values in $\Omega.$  First we will  show that
if
  $ \rho $ is a representing local
inverse for $\mathcal{V}^*(\Phi , \Omega )$, then
  $\rho$ is admissible.
By Proposition \ref{conn} $\Omega
\setminus\!\Phi ^{-1}(\Phi  (\overline{Z}))$ is connected. Since $\Phi$ is
holomorphic on $\overline{\Omega}$,   $ \Phi^{-1}
(\Phi(\overline{Z}) )$ is relatively closed in $\Omega.$ Thus $$
\Omega\setminus\!\overline{\Phi^{-1} (\Phi(Z))  } =\Omega\setminus\!\Phi ^{-1}(\Phi
(\overline{Z})).$$

Let
  $ \rho $ be a representing local inverse   for $\mathcal{V}^*(\Phi ,\Omega)$.
   Letting $$A_0=\Phi ^{-1}(\Phi  (\overline{Z})),$$
we will see that $\rho$ is admissible with respect to $A_0$.  Proposition \ref{conn} gives that $\Omega\setminus\!A_0$ is
connected. Since  $ \rho $ is a representing local inverse   for
$\mathcal{V}^*(\Phi ,\Omega)$,
  there are an operator $S$ in $\mathcal{V}^*(\Phi ,\Omega)$ and
    an open ball $\Delta$  such that 
 \begin{equation}
  Sg(w) = \sum_{j=1}^N c_jg(\rho_j(w))(J\rho_j)(w) ,\, w \in \Delta, g\in L_a^2(\Omega), \label{48}
\end{equation}
where  $\rho_{j_0}=\rho $ for some integer $j_0$ and  $c_{j}\neq
0$. This can be reformulated as
$$S^*K_w  = \sum_{j=1}^N \overline{c_j}\overline{(J\rho_j)(w)}K_{\rho_j(w)},\, w \in \Delta,$$
where $K_w$ denotes the reproducing kernel of  $L_a^2(\Omega)$ at $w
\in \Omega$. For each point $\lambda \in \Omega\setminus\!A_0$, there is an
open ball $\Delta_\lambda$ containing $\lambda$ where  a similar
representation as (\ref{48})   holds for $S$ and hence we get
 \begin{equation}
S^*K_w  = \sum_{j=1}^{N_\lambda}
\overline{c_j^\lambda}\, \overline{(J\rho_j^\lambda)(
w)}K_{\rho_j^\lambda(w)}, \,w \in \Delta_\lambda.  \label{rep}
\end{equation}
 Let $\gamma$ be   an arbitrary curve in $  \Omega\setminus\!A_0$ and $\gamma(0)\in \Delta. $
Since the union of  all these  open balls $\Delta_\lambda$ covers
$\gamma,$ applying the Henie-Borel theorem shows that there exist
finitely many such balls  whose union covers $\gamma$. If the
intersection of  two such balls is not empty,
then by the uniqueness of the representation (\ref{rep}) 
   we get direct  continuations of the tuple $\{\rho_j^\lambda :1\leq j \leq N\}$
and all $N_\lambda$ are equal. Then it follows that $\rho$ admits
analytic continuation along $\gamma$. Furthermore, $$
c_j^\lambda\neq 0, \, 1\leq j \leq N_\lambda,$$  as we showed
before,     for these $j$ the images
$\rho_j^\lambda(\Delta_\lambda)$   lie in
$\Omega.$ Thus the images of $\rho$ and its continuations lie in
$\Omega.$
 By arbitrariness of $\gamma $,  $\rho$ is admissible with
respect to $A_0$.  So a representing local inverse   for
$\mathcal{V}^*(\Phi ,\Omega)$ is admissible.

\vskip2mm

 For each point in  $\Omega\setminus\!A_0$,
there is a neighborhood where all members in $[\rho]$ are
holomorphic. As   $\Omega\setminus\!A_0$ is connected, if $\rho$ is a
representing local inverse for $\mathcal{V}^*(\Phi , \Omega )$, then
the number $k([\rho])$ of different members  in $[\rho]$
 (defined on a same domain $\Delta$)
does not depend on the choice of the domain   $\Delta$.  In this
sense, we denote this integer $k([\rho])$ by $\sharp [\rho]$, called
\emph{the multiplicity} of $[\rho]$ \cite{GH3}. Now fix a
representing local inverse $\rho$  of $\Phi  $. As done in
\cite{DSZ} or \cite{GH3}, define
\begin{equation}\label{def}   \mathcal{E}_{[\rho]}h(w)=  \sum_{ \sigma\in [\rho]}
h\circ \sigma (w) \,   J\sigma  (w),\,
 w\in \Omega\setminus\!A_0,
 \end{equation} where $h$ is an arbitrary function over
$\Omega\setminus\!A_0 $  or $\Omega.$ In the case of $\Phi $ being holomorphic
over $\overline{\Omega},$   the right hand side of (\ref{def}) is a
finite sum. Also by the above paragraph  $\sigma(z)\in \Omega\setminus\!A_0 $
if $z\in \Omega\setminus\!A_0$ and $\sigma\in [\rho]. $ Then the
formula (\ref{def}) makes sense. 
 For a local inverse $\rho$ of $\Phi ,$  let $\rho^-$
  denote  the inverse of $\rho$.

\vskip2mm
  Next, we will see that if $\rho$
is admissible, then   both  $\mathcal{E}_{[\rho]} $ and
$\mathcal{E}_{[\rho^-]} $  are in
 $\mathcal{V}^*(\Phi ,\Omega)$, and
$\mathcal{E}_{[\rho]}^*=\mathcal{E}_{[\rho^-]}$. Also,  one   will
see that if   $\rho$ is admissible,
   $\rho $ is  representing   for $\mathcal{V}^*(\Phi ,\Omega)$.
   In fact,   the proof of the theorem in
 \cite[p. 526]{T1} shows that the class of all admissible local inverses of $\Phi $
  is closed under composition;    if  $\rho $ is an  admissible
   local inverse, then its inverse   $\rho^-$ is also admissible.
Suppose that $\rho$ is an admissible local inverse of $\Phi  $ with
respect to $A_0$, defined as above.  %
 By the proof of \cite[Lemma 6.3]{GH3},
$ \mathcal{E}_{[\rho]}$ maps each function in $L_a^2(\Omega)$ to a
function in  $L_a^2(\Omega\setminus\!A_0)$; and furthermore, there
exists a constant $C$  such that
$$\| \mathcal{E}_{[\rho]} g\| \leq C \|g\| ,\, g\in L_a^2(\Omega).$$
Theorem \ref{removst2} says that $A_0$ is $L_a^2$-removable in $\Omega$, and
thus   $$L_a^2(\Omega\setminus\!A_0)= L_a^2(\Omega ).$$ So $
\mathcal{E}_{[\rho]}$  defines a bounded operator on $L_a^2(\Omega )$. Again by the proof of \cite[Lemma 6.3]{GH3},
we get
   $$\mathcal{E}_{[\rho]}^*=\mathcal{E}_{[\rho^-]}.
$$   Since both $\mathcal{E}_{[\rho]} $  and $\mathcal{E}_{[\rho^-]}
$ commute with $M_{\Phi}=\{M_{\phi_j }:1\leq j \leq
d\}$, they are in $\mathcal{V}^*(\Phi ,\Omega)$. This shows that
  $\rho$ is a representing local inverse  for $\mathcal{V}^*(\Phi ,\Omega)$.
Thus, all admissible local inverses are representing for
$\mathcal{V}^*(\Phi ,\Omega)$.

\vskip2mm
 Finally, we will  derive a delicate form of  (\ref{2.4}).
If $S$ is in $\mathcal{V}^*(\Phi ,\Omega)$,   it has the form as
(\ref{2.4})
 \begin{equation*}
  Sg(w) = \sum_{j=1}^N c_jg(\rho_j(w))(J\rho_j)(w) ,\, w \in \Delta,\, g\in L_a^2(\Omega),
\end{equation*}
where $\Delta$ is a subdomain of $\Omega$.     By applying
techniques of analytic continuation,
 if $\rho_k$ and $\rho_l$ lie in the same equivalent class, then their coefficients are equal \cite{DSZ}; that is,
 $c_k=c_l.$ To do this, let $\rho_l$ be the analytic continuation of $\rho_k$
 along a loop $\gamma,$ and for each $\rho_j$,   let $\widetilde{\rho_j}$ denote the
 analytic continuation of $\rho_j$
 along   $\gamma.$
 Then we get
 \begin{equation*}
  \sum_{j=1}^N c_jg(\rho_j(w))(J\rho_j)(w) =
  \sum_{j=1}^N c_jg(\widetilde{\rho_j}(w))(J\widetilde{\rho_j})(w)  ,\, w \in \Delta,\, g\in L_a^2(\Omega).
\end{equation*} By the uniqueness of coefficients $c_j$   and noting
$\widetilde{\rho_k}=\rho_l$  we have $c_k=c_l $, as desired. By
arbitrariness of $\Delta$, we can rewrite $S$ as
 \begin{equation}
  Sg(w) = \sum_{\rho}  c_\rho \mathcal{E}_{[\rho]} g  (w) ,\, w \in \Omega\setminus\!A_0 ,\, g\in L_a^2(\Omega). \label{2.5}
 \end{equation}
  Hence  each  operator $S$ in $\mathcal{V}^*(\Phi ,\Omega)$  can be
  represented as a linear span of $\mathcal{E}_{[\rho]} $,
  where $\rho$ are representing local inverses  for $\mathcal{V}^*(\Phi
  ,\Omega)$. Also, we have shown that
   those $\rho$ are exactly admissible
 local inverses  of $\Phi ,$ and each
  $\mathcal{E}_{[\rho]} $ is a well-defined bounded operator in  $\mathcal{V}^*(\Phi ,\Omega)$ to complete the proof of  Theorem \ref{thm11}.
 $\hfill \square$
\vskip2mm

Recall that for a local inverse $\rho$ of $\Phi ,$  $\rho^-$ denotes
the inverse of $\rho$. The proof of  Theorem \ref{thm11} gives the
following result.
\begin{prop}Suppose both $\Omega$ and  $\Phi $ satisfy the
assumptions in Theorem \ref{thm11}.  Then $ \rho $ is a representing
local inverse   for $\mathcal{V}^*(\Phi , \Omega )$ if and only  if
$\rho$ is admissible. In this case, both     \label{63}
$\mathcal{E}_{[\rho]} $  and $\mathcal{E}_{[\rho^-]} $  are in
 $\mathcal{V}^*(\Phi ,\Omega)$, and
$\mathcal{E}_{[\rho]}^*=\mathcal{E}_{[\rho^-]}$.
\end{prop}
  \noindent In particular,  $\rho$ is admissible  if and only
if $\rho^-$ is admissible.
\vskip1mm

We immediately have the following corollaries.

\begin{cor} Suppose both  $\Omega$ and  $\Phi $ satisfy the assumptions in Theorem \ref{thm11}.
 Then the dimension of $\mathcal{V}^*(\Phi ,\Omega)$ equals the number of
   equivalent classes of admissible local inverses of $\Phi $ on $\Omega $ .
    \end{cor}
 \begin{cor} Let $\Omega$ be a domain  in Theorem \ref{thm11}.
For  $n$ $(n\geq d)$  holomorphic functions  $\phi _1,\cdots, \phi
_n$ on $\overline{\Omega}$, if there are $d$ members $\phi
_{i_1},\cdots,\phi _{i_d}$ among them such that
 $$J(  \phi _{i_1},\cdots,\phi _{i_d}  )  \not\equiv 0$$
    on $\Omega$, then
  $\mathcal{V}^*(\phi _1,\cdots, \phi _n,\Omega) $ is of finite dimension, and  $\mathcal{V}^*(\phi _1,\cdots, \phi _n,\Omega) $
   is generated by $\mathcal{E}_{[\rho]}$, where $\rho$ are admissible local inverse of
   $\phi _1,\cdots, \phi _n$. \end{cor}
 On the other hand, Theorem \ref{thm11} may fail if the assumption on
 $\Phi$  is not satisfied.
  \begin{exam}
   Let $p(z_1,z_2)=z_1z_2$ and $\Omega=\mathbb{B}_2$.
  Then  $\dim \mathcal{V}^*(p,\mathbb{B}_2 )=\infty$
as $M_p$ has infinitely many pairwise orthogonal reducing
\label{interior} subspaces:
$$\overline{span\, \{p^kz_1^n: k=0,1,\cdots\}}, \, n\in \mathbb{Z}_+.$$
Let $\Phi =(p^2, p^3)$. Thus the image $\Phi(\Omega) $ of $\Phi $
is contained in
$$\{(z^2,z^3): z\in \mathbb{D}\},$$
and  thus $\Phi(\Omega) $ has no interior point.  Since
$$\mathcal{V}^*(p,\mathbb{B}_2 )\subseteq \mathcal{V}^*(\Phi
,\mathbb{B}_2 ),$$
 $\dim \mathcal{V}^*(\Phi ,\mathbb{B}_2 )=\infty.$
 \end{exam}

 However, for a single polynomial $q$ it may happen that  
$\dim \mathcal{V}^*(q,\Omega ) <\infty $ and even
$\mathcal{V}^*(q,\Omega )=\C I $.
 For instance, there exist abundant polynomails $q$ of degree one such that
$ \mathcal{V}^*(q, \mathbb{D}^d )=\C I $ for $d\geq 1$ \cite[Theorem 5.1]{WDH}.  Also, it is shown that for positive integers $k$ and
$l$,
 $2\leq \dim \mathcal{V}^*(z^k+w^l, \mathbb{D}^d)<\infty$  \cite[Theorem
1.1]{DH}.

  \section{\label{sectfif}Proofs of Theorems  \ref{mainproper} and \ref{41}}

In this section  we will present the proofs of Theorems  \ref{mainproper}
and \ref{41} and assume that $\Omega$ is  a bounded domain in $\mathbb{C}^d$.
\vskip2mm


 \noindent \textbf{Proof of Theorem \ref{mainproper}.}
 First we show that each local inverse $\sigma$ of $\Phi $ is
admissible  in $\Omega.$
 To do this,  let  $$\mathcal{E}=\Phi ^{-1}(\Phi  (Z)).$$
 By Corollary \ref{removst3}, $\mathcal{E}$ is relatively closed  and      $L_a^2$-removable in $ \Omega $.
   We will show that
  for each curve $\gamma \subseteq \Omega\setminus\!\mathcal{E}$, $\sigma$ admits analytic continuation with values in $\Omega.$
In fact,   by Proposition \ref{conn}  $\Omega\setminus\!\mathcal{E}$ is
connected. Given a curve $\gamma$ in $\Omega\setminus\!\mathcal{E}$,  Theorem
\ref{variety} shows that for each point $\lambda$
 on $\gamma$  there exists an enough small ball $B_\lambda$ centered at $\lambda$ such that
\begin{equation}\label{add02}
\Phi ^{-1}(\Phi (B_\lambda))=\bigsqcup_{j=1}^n U_j(\lambda),
\end{equation}
where $U_1(\lambda)=B_\lambda,$ and  $\{U_j(\lambda)\}_{j=1}^{n} $ are disjoint domains on which
$\Phi $ is biholomorphic. 
This integer $n$ only depends on $\Phi.$
  Then it is easy to define $n$ local inverses of $\Phi :$
 $$\rho_1^\lambda, \rho_2^{\lambda}, \cdots, \rho_n^{\lambda},$$
 which map $B_\lambda$ bijectively  to $U_1(\lambda),\cdots,U_n(\lambda)$, respectively.
 Since $\gamma$ is compact, there are finitely may balls $\{B_{\lambda_k}\}_{k=1}^{N}$ whose union covers $\gamma.$
 After reordering them, we may require that
 $$B_{\lambda_k}\cap B_{\lambda_{k+1}}   \neq \emptyset,\, k=1,\cdots, N-1.$$
 Clearly, those local inverses $\{ \rho_j^{\lambda_k}\}_{  j=1}^{n} $ on $B_{\lambda_k}$
 are direct   continuations of $\{ \rho_j^{\lambda_{k+1}}\}_{j=1}^{n}$ defined on $B_{\lambda_{k+1}}$,
  up to a permutation. Thus each local inverse admits analytic continuation along $\gamma $ with values in $\Omega,$ as  desired.
So by arbitrariness of $\gamma $ all
local inverses of $\Phi $ are admissible.

   We  claim that   each  operator $S$ in $\mathcal{V}^*(\Phi ,\Omega)$  can be
 represented as a linear span of $\mathcal{E}_{[\rho]} $,
 where $\rho$ are local inverses of $\Phi$.
 For this,  recall that
$$ \mathcal{E}_{[\rho]}h(w)=  \sum_{ \sigma\in [\rho]} h\circ \sigma (w)
 J\sigma (w), w\in \Omega\setminus\! \mathcal{E},  \, h\in L_a^2(\Omega),$$
 and  we have just shown that
 for each local inverse $\rho,$  $\rho$ is admissible.
  Using Corollary \ref{removst3} and   the discussions of the paragraph below (\ref{def}),
    we have that both $\mathcal{E}_{[\rho]}$ and $\mathcal{E}_{[\rho^-]}$
 are   bounded operators in $\mathcal{V}^*(\Phi ,\Omega) $.
 Besides, for a fixed point $\lambda\in \Omega\setminus \mathcal{E}$ with
  $\mathcal{E}=\Phi^{-1}(\Phi(Z)),$ we have (\ref{add02})
$$\Phi ^{-1}(\Phi (B_\lambda))=\bigsqcup_{j=1}^n U_j(\lambda).$$
This, along with the discussions below (\ref{212}), enable us
 to obtain a similar formula as (\ref{2.4}) for each operator $S$ in $ \mathcal{V}^*(\Phi ,\Omega) $,
$$
  Sg(w) = \sum_{j=1}^N c_jg(\rho_j(w))(J\rho_j)(w) ,\, w \in \Delta , g\in L_a^2(\Omega),
$$
where $\Delta$ is a neighborhood of $\lambda.$
By the technique of analytic continuation, we get
that if $\rho_i$ and $\rho_k$ lie in the same equivalence $[\rho_i]$, then $c_i=c_k.$
This means that $S$ can be represented as a linear span of    $ \mathcal{E}_{[\rho]}$, where $\rho$ are local inverses of $\Phi$ as desired.



Recall
 $$\mathcal{S}_{\Phi }= \{(z,w)\in \Omega^2:  \Phi (z)=\Phi (w),\, z\not\in \Phi ^{-1}(\Phi  ({Z}))\}.$$
  To complete the proof of  Theorem \ref{mainproper}, we need  show that
the number of    equivalent classes of   local inverses of $\Phi $
equals the number of components of $\mathcal{S}_{\Phi }$. To do
this,  first we have two observations. On one hand,   each point $(z,w)$
in  $\mathcal{S}_{\Phi }$ has the form $(z,\rho(z))$ for some local
inverse $\rho$ of $\Phi ,$ and there is a neighborhood $\mathcal{O}$
of  $(z,\rho(z))$ such that each point in $\mathcal{O}\cap
\mathcal{S}_{\Phi }$ has the form $(\lambda,\rho(\lambda))$. On the other hand,
for two local inverses $\rho $ and $\sigma$ of $\Phi ,$ they are
equivalent if and only if there is a curve $\gamma$ along which
$\rho$ admits analytic continuation
  $\sigma$. The curve $\gamma$ gives a curve    in  $\mathcal{S}_{\Phi }$ to joint $(z,\rho(z))$  and $(w,\sigma(w))$.
Thus if $\rho $ is equivalent to $\sigma$, then
  their images lie in the same component of $\mathcal{S}_{\Phi }$.  By  a similar discussion
  the converse is also true. So
   the number of    equivalent classes of   local inverses of $\Phi $  equals the number of
components of $\mathcal{S}_{\Phi }$. This completes the proof.
 $\hfill \square$  

\vskip2mm
In general, if $\Phi:\Omega \to \mathbb{C}^d $ is holomorphic
on a neighborhood of $\overline{\Omega}$, it is likely that not
all local inverses of $\Phi$ are admissible.
Most of them appear to be ``local" rather than ``global"
 (the precise term is ``admissible"). However,
all local inverses of a holomorphic proper map are  global, as shown in
 the proof of Theorem \ref{mainproper}. Some words are in order.

 \begin{rem}
 \begin{itemize}
	\item[(1)] Suppose $\Phi$ is a holomorphic proper map from
 $\Omega$ to $\Phi (\Omega)$. Then for each curve $\gamma$ in $\Omega$, $\Phi^{-1}(\Phi (\gamma))$
  is contained in $\Omega$ and hence for each local inverse $\rho$ of $\Phi$,
 its analytic continuation along $\gamma$ must lie in $\Omega$ (see Theorem \ref{thm11}).
 This enables us to drop  the condition on  the domain $\Omega$ that the interior points of the closure of $\Omega$ equals $\Omega$.
	\item [(2)] In the proof of Theorem \ref{thm11}, we assume that
   $\Phi$ is
holomorphic on $\overline{\Omega}$ and in that case    $ \Phi^{-1}
(\Phi(\overline{Z}) )$ is relatively closed in $\Omega.$ Thus
 \begin{equation}\label{add01}
\Omega\setminus\!\overline{\Phi^{-1} (\Phi(Z))  } =\Omega\setminus\!\Phi ^{-1}(\Phi
(\overline{Z})).\end{equation}
Later, we do analytic continuation of local inverses of $\Phi$
in $\Omega\setminus\!\overline{\Phi^{-1} (\Phi(Z))  }$.
In the situation where  $\Phi$ is a holomorphic proper map on $\Omega$,  we need verify a similar statement
as (\ref{add01}),
  $$ \Omega\setminus\!\overline{\Phi^{-1} (\Phi(Z)) }=\Omega\setminus\! \Phi^{-1} (\Phi(Z))   . $$For this, just note that by Corollary \ref{removst3}  $\Phi^{-1} (\Phi(Z))$ is relatively closed in  $\Omega$, which forces the above equality to hold.
\item [(3)] The proof of Corollary \ref{removst3} gives that for a holomorphic proper map $\Phi$
over $\Omega$ and a zero variety $E$ of $\Omega$, $\Phi^{-1} (\Phi(E)) $ is locally contained in a zero variety.
 Thus, for  a holomorphic proper map $\Phi$ on $\Omega$ and a local inverse $\rho$ of $\Phi$,
  there is a relatively closed
 subset $A$ of $\Omega $ such that $A$ is locally contained in a zero variety
and for each curve $\gamma$ in $\Omega\setminus\!A$, $\rho$ admits analytic
continuation with values in $\Omega.$
	 \end{itemize}
\end{rem}
We proceed to give the proof of Theorem \ref{41}.
 \vskip2mm
 \noindent \textbf{Proof of Theorem \ref{41}.} If $\Phi $ is    biholomorphic,   the only local inverse of $\Phi $
  is the identity map and hence  by the proof of  Theorem \ref{thm11},
    (\ref{213}) will become $\Delta=\Phi^{-1}(\Phi(\Delta)).$ The resulting formula (\ref{2.4})
 $$
  Sg(w) = \sum_{j=1}^N c_jg(\rho_j(w))(J\rho_j)(w) ,\, w \in \Delta , g\in L_a^2(\Omega)
$$  contains exactly one term; that is, each operator $S$ in $ \mathcal{V}^*(\Phi ,\Omega) $
is a constant tuple of the identity. Thus
  $ \mathcal{V}^*(\Phi ,\Omega) $ is trivial.

 Conversely
we will show that if $\Phi  :\Omega\to \Omega'  $ is a    non-biholomorphic proper map then
   $ \mathcal{V}^*(\Phi ,\Omega ) $ is nontrivial.
 Since $\Phi $  is not biholomorphic,
  Theorem \ref{variety} gives that   there exists a nontrivial
 local inverse $\rho$ of $\Phi $. Recall that
$$ \mathcal{E}_{[\rho]}h(w)=  \sum_{ \sigma\in [\rho]} h\circ \sigma (w)
  J\sigma (w), w\in \Omega\setminus\! \mathcal{E}, \, h\in L_a^2(\Omega).$$
  Using Corollary \ref{removst3} and   the discussions of the paragraph below (\ref{def}),  we have that both $\mathcal{E}_{[\rho]}$ and $\mathcal{E}_{[\rho^-]}$
 are   bounded operators in $\mathcal{V}^*(\Phi ,\Omega) $.
  Since $\rho$ is not the identity map,  $\mathcal{E}_{[\rho]}$ is not
  a scalar multiple of the identity operator.
   This implies that   $ \mathcal{V}^*(\Phi ,\Omega) $ is nontrivial
    to complete the proof of Theorem \ref{41}.
 $\hfill \square$  

  \section{\label{six} Some examples  of   $\mathcal{V}^*(\Phi , \Omega)$ }

\vskip2mm

 In  this section,   we  will show   some examples of $\mathcal{V}^*(\Phi , \Omega)$.
   In \cite{DPW},
   Douglas, Putinar and Wang showed     that $\mathcal{V}^*(B, \mathbb{D}) $ is abelian if $B $ is a finite
Blaschke product.
 Hence by  Thomson's commutant theorem on multiplication operators \cite{T1}, if $\phi $ is   holomorphic  on the
closed unit disk $\overline{\mathbb{D}}$ and $\phi $ is not
constant, then there is a finite Blaschke product $B$ such that
 $\mathcal{V}^*(\phi , \mathbb{D})=\mathcal{V}^*(B, \mathbb{D})$ is always abelian.
 But Example \ref{exam6.1} shows that $\mathcal{V}^*(\Phi ,\Omega)$ may be not abelian  in multi-variable case and Examples \ref{exam6.3} and \ref{exam6.4} show that  $\mathcal{V}^*(\Phi ,\Omega)$ may be  abelian and nontrivial in multi-variable case also.

  \vskip2mm

  For each map $\Phi :\Omega\to \C^d$,
  the deck transformation group  $G(\Phi )$ of $\Phi $ consists of all  holomorphic automorphism $\rho$
  of $\Omega$ satisfying $\Phi  \circ \rho =\Phi .$
  If all local inverses of a holomorphic map $\Phi :\Omega\to \C^d$ lie in $G(\Phi )$,
  then $\Phi $ is called a regular map \cite{Mi}.   
  One also say that the  proper holomorphic map $\Phi$ is factored
  by automorphisms \cite{BeD,DSe}
  if  $\Phi$ is regular and the deck transformation group  $G(\Phi )$ is finite.
  For more information of such groups, we
  call the reader's attention to \cite{BeD,DSe}.
  An extensive study was made in \cite{BDGS},  and
  especially in Sections 5 and 6 there, for the joint reducing subspaces of multiplication operators
  defined by a class of proper holomorphic maps, whose deck transformation groups are generated by
  so called pseudoreflections on $\mathbb{C}^n$.

  It is known that
  a finite Blaschke product is regular if and only if it is of the form
  $$m_1\circ \varphi \circ m_2,$$
  where both $m_1$ and $m_2$ are in  Aut($\mathbb{D}$), 
  and $\varphi(z)=z^n$ for some positive integer $n$ \cite{GH3}. However, in multi-variable case we
  will see more examples induced by  polynomials or those arising from finite reflection groups\cite{BeD}.
  Given a discrete group $\Gamma$,  recall that the group von
  Neumann algebra   $\mathcal{L}(\Gamma)$
  is the weak closure of the linear span of all
  left regular representations $\{L_\rho: \rho\in \Gamma \}$
  of $\Gamma$ on $l^2(\Gamma)$, defined by
  $$L_\rho g(\sigma)=g(\rho^{-1} \sigma), \sigma \in \Gamma, \, g\in l^2(\Gamma).$$
  Theorem \ref{mainproper} has an immediate corollary.
  \begin{cor}  \label{corproper} Suppose $\Phi   $ is
  	a holomorphic  regular proper map   on $ \Omega $.  Then $\mathcal{V}^*(\Phi ,\Omega)$ is
  	$*$-isomorphic to $\mathcal{L}(G(\Phi ))$, where $G(\Phi )$ is the deck transformation group
  	of $\Phi .$
  \end{cor}                  
  \begin{proof} Suppose $\Phi   $ is
  	a holomorphic  regular proper map   on $ \Omega $.
  	Then each  local inverse of $\Phi $ corresponds to a member in  $G(\Phi )$.
  	Let  $$G(\Phi )=\{\rho_j: 1\leq j \leq
  	n\}.$$
  	For each  local inverse $\rho $  of
  	$\Phi $, $\rho$ is in $ G(\Phi )$, and thus
  	$\mathcal{E}_{[\rho]}$ defines a unitary operator in
  	$\mathcal{V}^*(\Phi ,\Omega)$.
  	Letting $\mathcal{E}_\rho  $ denote  $\mathcal{E}_{[\rho]}$,  Theorem \ref{mainproper} gives that $\mathcal{V}^*(\Phi ,\Omega)$
  	is generated by   $\mathcal{E}_{\rho}$ for $\rho$ in $ G(\Phi ).$
  	Noting $$\mathcal{E}_\rho^*\mathcal{E}_\sigma^*= \mathcal{E}_{\rho\circ \sigma}^*,\, \rho,\sigma\in G(\Phi ) ,$$
  	we have that the map given by
  	$$\sum_{j=1}^n c_j \mathcal{E}_{\rho_j}^* \mapsto \sum_{j=1}^n c_j L_{\rho_j}, $$
  	is  a $*$-isomorphism from  $\mathcal{V}^*(\Phi ,\Omega)$ to  $\mathcal{L}(G(\Phi ) ).$
  \end{proof}
  Note that the assumption of Corollary \ref{corproper} can be
  reformulated as:  $\Phi   $ is  a
  holomorphic proper map and each local inverse of $\Phi $ is holomorphic in $ \Omega
  $.
  To see this, suppose $\rho$ is a  local inverse of $\Phi $,
  then so is its inverse $\rho^-.$
  Since both $\rho\circ \rho^-$ and  $\rho^-\circ \rho $ locally are
  identity, and all local inverses of $\Phi $ are holomorphic in $ \Omega
  $, it follows that both $\rho$ and $\rho^-$ are holomorphic
  automorphisms of $ \Omega$, forcing $\rho\in G(\Phi )$. Thus  $\Phi   $ is  regular, as desired.

  \begin{exam}\label{exam6.1} Let $\Omega$ be $\mathbb{D}^2$ or $\mathbb{B}_2$. Let
$$p(z_1,z_2)=z_1^2+z_2^2 \quad \mathrm{and} \quad q(z_1,z_2)=z_1^2 z_2^2 ,$$
and $$\Phi  =(p,q) .$$
 There are exactly eight admissible local inverses of $ \Phi  $:
$\rho_1,\cdots, \rho_8$, which are defined by
$$\rho_1(z_1,z_2)=(z_1,z_2),  \,\rho_2(z_1,z_2)=(-z_1,z_2), \, \rho_3(z_1,z_2)=( z_1,-z_2),$$ \,
$$\rho_4(z_1,z_2)=( -z_1,-z_2); \rho_5(z_1,z_2)=(z_2,z_1) , \,\rho_6(z_1,z_2)=(-z_2,z_1),$$
and
$$     \rho_7(z_1,z_2)=(z_2,-z_1),\,
  \rho_8(z_1,z_2)=( -z_2,-z_1).$$
Each $\rho_j$   induces a unitary operator $U_j$ on  \label{nonabel}
$L_a^2(\Omega):$
$$U_jf=f\circ \rho_j,\,f\in L_a^2(\Omega).$$
By   Theorem \ref{thm11}, the von Neumann algebra
$\mathcal{V}^*(\Phi ,\Omega)$ is generated by $\{U_j: 1\leq j\leq 8\}.$
Noting that the deck transformation group $G(\Phi )$ equals $\{\rho_j:
1 \leq j\leq 8\},$
we have that $\mathcal{V}^*(\Phi ,\Omega)$ is not abelian since $G(\Phi )$ is not abelian.
 \end{exam}

    The following proposition
tells us that the map $\Phi $ in Example \ref{nonabel} is a holomorphic
proper map.
 \begin{prop}  Let $\Omega  $ be  $\mathbb{D}^2$ or $\mathbb{B}_2$, and let $\Phi(z_1,z_2) =(z_1^2+z_2^2, z_1^2z_2^2)$.
  Then $\Phi (\Omega)$ is open and $\Phi  $ is a  holomorphic proper map from    \label{proper1}
  $\Omega$ to $\Phi (\Omega)$. \end{prop}
 \begin{proof} Let $\Psi(z_1,z_2)=(z_1 +z_2 , z_1 z_2 ) $. First we will   show that $\Psi$ is an open map in $\C^2$; that is,
 $\Psi$ maps  each open ball to  an open set. If so,
 noting that $\Phi (z_1 ,z_2)=\Psi  (z_1^2,z_2^2)$,
  $\Phi $ is also an open map. To show  that $\Psi$ is open,  we just need show that for a given  point $w=(w_1,w_2)$  and for any neighborhood $U$ of $w$, $\Psi(U)$ contains an open ball centered at $\Psi(w).$
   Let $\Psi(w)=(s_1,s_2)$.  Noting that $w_1$ and $w_2$ are zeros of the polynomial $v$ defined by
   $$v(x)=x^2-s_1x+s_2, $$
   we can find   two disks $D_1$ and $D_2$  centered at $w_1$ and $w_2$, respectively, such that
   $$D_1\times D_2 \subseteq U.$$
 If $w_1\neq w_2,$ we can require that $\overline{D_1} \cap
\overline{D_2} =\emptyset.$ By
  applications of Rouche's theorem on $\partial D_1$ and $\partial D_2$, respectively,   if $s'=(s_1',s_2')$ is enough close to  $\Psi(w) $,
 then the polynomial
 $u(x)=x^2-s_1'x+s_2'$ has exactly two zeros: one in $D_1$ and the other in $D_2.$   If $w_1= w_2,$
 then let $D_2=D_1$. By similar discussion, the polynomial
 $u(x)=x^2-s_1'x+s_2'$ has exactly two zeros in $D_1$. In either case,     $\Psi(D_1\times D_2)$   contains an open ball centered at $\Psi(w),$ and so does  $\Psi( U)$. Thus
 both $\Psi$ and $\Phi $ are open maps.

Next we show that $\Phi $ is a proper map. Noting
$\Phi $ is holomorphic on  $\om$,  we have that $\Phi $ is a proper map
 if and only if $\Phi (\partial \Omega) \subseteq \partial \Phi (  \Omega).$ Since $\Phi (\Omega)$ is open,   it suffices to show that
$\Phi (\partial \Omega) \cap \Phi (  \Omega)=\emptyset$. To do this,  fix
$\lambda=(\lambda_1,\lambda_2)$ and $\mu=(\mu_1,\mu_2)$. Let
$$x_1=\lambda_1^2+\lambda_2^2 \quad \mathrm{and}\quad x_2= \lambda_1^2 \lambda_2^2.$$
Thus $\lambda_1^2$ and $ \lambda_2^2$ are the solutions of the
equation
$$x^2-x_1x+x_2=0 \, (x\in \mathbb{C}).$$
So, $\Phi (\lambda)=\Phi (\mu)$ if and only if
$(\lambda_1^2,\lambda_2^2)$ is a permutation of $(\mu_1^2,\mu_2^2)$.
This is equivalent to that  there exists a member $\rho\in G(\Phi )$ (see Example \ref{nonabel})  satisfying
$$\mu=\rho(\lambda).$$
Since both  $\Omega$ and $\partial \Omega$ are invariant under the action of
$G(\Phi )$, if $\Phi (\lambda)=\Phi (\mu)$, we have that  $\lambda\in \partial
\Omega$ if and only if $\mu\in \partial \Omega$. This gives
$\Phi (\partial \Omega) \cap \Phi (  \Omega)=\emptyset$, to complete
the proof.
 \end{proof}

By the proof of Proposition \ref{proper1}, we deduce that if $\Omega $ is replaced
by any domain
 invariant under the deck transformation group
 $G(\Phi )$, then the same results hold. Similarly, one can prove that
 $(z_1+z_2,z_1^2+z_2^2)$  is a holomorphic
  proper map on either  $\mathbb{D}^2$ or $\mathbb{B}_2$.
  By the same idea, one can also show that for positive integers
$\alpha_1 , \cdots, \alpha_d $,
$(z_1^{\alpha_1},z_2^{\alpha_2},\cdots, z_d^{\alpha_d})$ defines
 a holomorphic proper map on $\mathbb{B}_d$. In these cases, a direct
 application of  Corollary \ref{corproper} shows that $\mathcal{V}^*(\Phi ,\Omega)$ is
 $*$-isomorphic to $\mathcal{L}(G(\Phi ))$.  
In fact, Proposition \ref{proper1} induces more holomorphic proper maps as below.
\begin{exam}\label{exam6.3} Suppose $\Phi =(z_1 z_2^2,z_1  +z_2^2 )$. Since  $\Phi $
 is the composition of  \linebreak $ (z_1 z_2,z_1  +z_2)$ and $ (z_1 , z_2^2 )$, $\Phi $
 is a holomorphic proper map on  $\mathbb{D}^2$.
  By Theorem
\ref{mainproper}, studying the structure of $\mathcal{V}^*(\Phi ,
\mathbb{D}^2)$ reduces to
  determining all admissible local inverses of $\Phi $ on  $\mathbb{D}^2$.
  To do so, we will solve the following equation in $\mathbb{D}^2$:
$$\Phi (w)=\Phi (z).$$
Following the proof of Proposition \ref{proper1}, we see that  $\Phi (z)=\Phi (w) $
if and only if  one of the following holds:
 \[    \left\{\begin{array}{cc}
   w_1 =z_1, \\
     w_2^2=z_2^2,
 \end{array}\right.
 \]
or
 \[    \left\{\begin{array}{cc}
  w_1=z_2^2, \\
   w_2^2=z_1.
 \end{array}\right.
 \]
 By solving these
equations, we get   three equivalent classes of admissible local
inverses   of $\Phi $:   $\rho_1(z_1,z_2)=(z_1,z_2)$, $\rho_2(z_1,z_2)=(z_1,-z_2)$ and   $\{\rho_3,\rho_4\}$    with      $$\rho_3 (z_1,z_2)=(z_2^2,\sigma(z_1)), \quad
\mathrm{and}\quad \rho_4 (z_1,z_2)=(z_2^2,-\sigma(z_1)), $$ where
$\sigma$ denotes the branch of $\lambda \mapsto \sqrt{\lambda}$
defined on a neighborhood of $\frac{1}{2} $ such that
$\sigma( \frac{1}{2})=\sqrt{\frac{1}{2}}$. Note that
 for each curve $\gamma \subseteq \mathbb{D}\backslash \{0\}$ with $\gamma(0)= \frac{1}{2}$,
 $\sigma$ always admits an analytic continuation along $\gamma$ and $-\sigma$ lies in the same equivalence with
 $\sigma $; that is, $-\sigma\sim \sigma.$
 These local inverses naturally derives three operators in
$\mathcal{V}^*(\Phi ,   \mathbb{D}^2)$: $I,$ $S_1$ and $S_2$, defined by   $$S_1 f(z_1,z_2)= f(z_1,-z_2),\,
(z_1,z_2)\in \mathbb{D}^2,$$ and
$$S_2 f  =J\rho_3 \cdot f\circ \rho_3 + J\rho_4 \cdot f\circ \rho_4,$$ for $f\in L_a^2(\mathbb{D})$. Formally,
$$S_2 f(z_1,z_2)= \frac{z_2}{\sigma(z_1)}[- f(z_2^2,    \sigma(z_1))+  f(z_2^2, - \sigma(z_1))],
\, (z_1,z_2)\in (\mathbb{D}\setminus\!\{0\})\times  \mathbb{D}, $$
for   $f\in L_a^2(\mathbb{D}^2).$
We emphasize that $S_2f$ is locally defined first and then
it extends analytically to the whole bidisk $\mathbb{D}^2$.
 Thus
  $$\dim \mathcal{V}^*(\Phi ,   \mathbb{D}^2)=3.$$
 For a     finite dimensional von Neumann
 algebra $ \mathcal{A}$ on a Hilbert space $H$, $\mathcal{A}$ is $ *$-isomorphic to 
$\bigoplus_{k=1}^r M_{n_k}(\mathbb{C}) $ \cite[Theorem III.1.2]{Da}.  Thus $\mathcal{V}^*(\Phi ,   \mathbb{D}^2)$
  is $*$-isomorphic to $\mathbb{C}\oplus \mathbb{C} \oplus \mathbb{C},$ and then
   $\mathcal{V}^*(\Phi ,   \mathbb{D}^2)$
  is abelian.

  We have to point out that $\Phi $ is not a proper map on $\mathbb{B}_d$.
   On the other hand,  $\Phi $ has exactly two  admissible local inverses on $\mathbb{B}_d$: $(z_1,z_2)$, $(z_1,-z_2)$.
   Then by    Theorem  \ref{thm11},  $ \dim \mathcal{V}^*(\Phi ,
   \mathbb{B}_2)=2$, and then
    $\mathcal{V}^*(\Phi ,    \mathbb{B}_2)$
  is $*$-isomorphic to $\mathbb{C}\oplus  \mathbb{C}.$
\end{exam}
 To contrast with Example \ref{exam6.3} we have the following   interesting example.
\begin{exam}\label{exam6.4} Suppose $\Phi =(z_1^2 z_2^4,z_1^2  +z_2^4 )$. Then $\Phi $ is a holomorphic
proper map on  $\mathbb{D}^2$ but not on
$\mathbb{B}_2.$ By the argument in  Example \ref{exam6.3} we get
exactly  twelve  equivalent classes of admissible local inverses  of
$\Phi $ on $\mathbb{D}^2$: $$(z_1,i^kz_2), ~~(-z_1,-i^kz_2), ~~(1\leq
k\leq 4);$$
 $$(z_2^2, \pm  \sigma(z_1)), (-z_2^2, \pm  \sigma(z_1));
   (z_2^2, \pm i \sigma(z_1)), (-z_2^2, \pm  i \sigma(z_1)),$$
   where $\sigma$ is the local inverse defined in Example \ref{exam6.3}.
  By a simple computation,  not all these equivalent classes commute with each other under composition. Theorem \ref{mainproper}  gives that  $\mathcal{V}^*(\Phi ,   \mathbb{D}^2)$
is not abelian.

But there are only eight admissible local inverses  of $\Phi $ on
$\mathbb{B}_2 $: $(z_1,i^kz_2)$ and $(-z_1, i^kz_2)$$(1\leq k\leq
4)$, each corresponding to a unitary operator
 in $\mathcal{V}^*(\Phi ,   \mathbb{B}_2)$. Applying Theorem \ref{thm11} shows  that  $\mathcal{V}^*(\Phi ,   \mathbb{B}_2)$ is
 $*$-isomorphic to $\mathcal{L}(G(\Phi ,   \mathbb{B}_2))  $. Note that
 $ G(\Phi ,   \mathbb{B}_2)$ is abelian but not   cyclic.
\end{exam}
  In  the single-variable case, all known
 abelian deck transformation groups $G(\Phi  )$ are
$*$-isomorphic to   $\mathbb{Z}_n$ or $\mathbb{Z}$, a cyclic group.
For example, the deck transformation group  $G(z^n)$ of $z^n$ over
$\mathbb{D}$ is isomorphic to $\mathbb{Z}_n.$ If $\Phi (z)=\exp(-\frac{1+z}{1-z})$,
then $\Phi $ is a covering map from $\mathbb{D}$ onto $\mathbb{D}\setminus\!\{0\}$,
and $G(\Phi )$ is isomorphic to $\pi_1(\mathbb{D}\setminus\!\{0\}) \cong  \mathbb{Z} $. Example \ref{exam6.4} provides a different example.

The following example   comes from Examples 2 and 3 in \cite{BeD}.

\begin{exam} \cite{BeD} Let $\Omega_1$ be the domain
$$\Omega_1=\{(z_1,z_2)\in \mathbb{C}^2: |z_1|^2+|z_1|^{-2}+|z_2|^2+|z_2|^{-2}<6\}.$$
and $$F(z)=(z_1z_2+z_1^{-1}z_2+z_1z_2^{-1}+z_1^{-1}z_2^{-1}, z_1+z_1^{-1}+z_2+z_2^{-1}).$$
Then $F:\Omega \to F(\Omega)$ is a  holomorphic regular proper map.
 Furthermore,  the deck transformation group of $F$
 is  the dihedral group $G_1$ of order $8$ generated by two members
 $$(z_1,z_2)\mapsto (\frac{1}{z_1},z_2) \quad \mathrm{and} \quad
  (z_1,z_2)\mapsto (z_2,z_1) .$$ Then by Corollary \ref{corproper}
    $\mathcal{V}^*(F ,  \Omega_1) $  is
 $*$-isomorphic to $\mathcal{L}(G_1)  .$ It is clear that $\mathcal{V}^*(F ,  \Omega_1) $
 is not abelian.

 Let $\Omega_2$ be the domain
$$\Omega_2=\{(z_1,z_2)\in \mathbb{C}^2: |z_1|+|z_1|^{-1}
+|z_1^2z_2^{-1}|+|z_1^{-2}z_2|+|z_1z_2^{-1}|+|z_1^{-1}z_2|  <8\}.$$
Then the automorphisms of $\Omega_2$
$$(z_1,z_2)\mapsto  (z_2/z_1,z_2)  \quad \mathrm{and} \quad (z_1,z_2)\mapsto  ( z_1,z_1^3/z_2)  $$
generate a dihedral group $D_6$ of order $12$, which is the deck transformation group of
the holomorphic regular proper map $H=(h_1, h_2):\Omega_2\to H(\Omega_2)$ where
$$h_1(z)=z_1+z_1^{-1}+z_1^{-1}z_2+z_1z_2^{-1}+z_1^2z_2^{-1}++z_1^{-2}z_2,$$
and
$$h_2(z)=z_2+z_2^{-1}+z_1^{3}z_2^{-1}+z_1^{-3}z_2+z_1^{-3}z_2^{2}+ z_1^{3}z_2^{-2} .$$
  Then   $\mathcal{V}^*(H ,  \Omega_2) $  is   non-ableian since
Corollary \ref{corproper} implies that
    $\mathcal{V}^*(H ,  \Omega_2) $  is
    $*$-isomorphic to $\mathcal{L}(D_6)  .$
 \end{exam}
\vskip2mm
Inspired by Proposition \ref{proper1}, we  consider the permutation
groups $S_d$ ($d\geq 2$). Let
$$\phi _1=\sum_{1\leq j\leq d}z_j,$$
$$\phi _2=\sum_{1\leq i<j\leq d }z_iz_j,$$
$$\cdots $$
and $\phi _d=z_1z_2\cdots z_d$ \cite[Chapter 8]{GH4}. Also put  $\Phi=(\phi_1,\cdots,\phi_d)$.
\begin{prop}\label{proper2}As defined above, $\Phi $ is an open map on $\mathbb{C}^d$.  Furthermore,
 $\Phi $ is a holomorphic proper map on $\Omega$
  if $\Omega$  is invariant under the action of  $S_d$.   \end{prop}
\begin{proof}%
First we will prove that $\Phi $ is an open map on $\mathbb{C}^d.$ To
see this, we first show that for each $\lambda$ in  $\mathbb{C}^d$,
$\Phi ^{-1}(\Phi (\lambda))$ is a finite set. This is equivalent that   $\Phi (\lambda)=\Phi (\mu)$ if and only if there is a permutation
$\rho$ on $\mathbb{C}^d$ such  that $\mu=\rho(\lambda).$
 In fact, let $\Phi (\lambda)=w$ and $w=(w_1,\cdots,w_d)$. Consider the following algebraic equation in $x$:
 $$x^d-w_1x^{d-1}+\cdots+(-1)^{d-1}w_{d-1}x+(-1)^{d}w_d=0.$$
 Let $Q$ denote the  polynomial of $x$ defined by the left hand side of this equation.
Since $\Phi (\lambda)=w$, $\lambda$ are   $d$ zeros of  $Q$,
counting multiplicity. Since $$\Phi (\lambda)=\Phi (\mu), $$
  $\mu$ are  also
$d$ zeros of $Q$. Thus up to a permutation $\mu$ equals
$\lambda $, as desired. So applying Theorem \ref{prop} shows that
$\Phi $ is an open map on $\mathbb{C}^d.$

Using the same argument as in the proof of  Proposition   \ref{proper1},
 one can show that $\Phi $ is a holomorphic proper map on $\Omega $   if $\Omega$  is invariant under the action of  $S_d$.
 \end{proof}
Let $\Phi $  be defined as above Proposition \ref{proper2}. The deck
transformation group $G(\Phi )$ of $\Phi $ is
 isomorphic to $ S_d$.   Since $\Omega$ is invariant under
 $S_d$,  by Corollary \ref{corproper} $\mathcal{V}^*(\Phi,\Omega )$ is $*$-isomorphic
to $\mathcal{L}(S_d)$. Therefore, $\mathcal{V}^*(\Phi,\Omega )$ is abelian
if and only if $S_d$ is abelian, if and only if $d =1,2.$ The
  case of  $\Omega=\mathbb{B}_d$ or $\mathbb{D}^d,$ this was
obtained by
 \cite[Proposition 8.4.6]{GH4}.

 \textbf{Acknowledgements.}
This work is  partially supported by NSFC(12071134; 12271090).



\begin{thebibliography}{wide-label}
	{\footnotesize
		
		



\bibitem{BeD}
E.~Bedford, J.~Dadok,
\emph{Proper holomorphic mappings and real reflection groups,}
J. Reine Angew. Math. 361(1985), 162-173.
\bibitem{Be}
S.~Bell, \emph{The Bergman kernel function and proper holomorphic mappings}, Trans. Amer. Math. Soc.,
\textbf{270}(1982), 685-691.
\bibitem{BDGS}
S.~Biswas, S.~Datta,G.~Ghosh, S.~Shyam Roy,  \emph{Reducing submodules of Hilbert modules and Chevalley-Shephard-Todd theorem},
Adv. Math. 403(2022), Paper No. 108366, 54 pp. 
\bibitem{BSMS}
S.~Biswas, G.~Ghosh, G.~Misra, S.~Shyam Roy,
\emph{On reducing submodules of Hilbert modules with  Sn -invariant kernels},
J. Funct. Anal. 276(2019), 751--784. 
\bibitem{Bj}
		A.~Bjorn, \emph{Removable singularities for weighted Bergman spaces},
		Czechoslovak Mathematical Journal, \textbf{56}(2006), 179-227. 
		\bibitem{Bo}
		S. Bochner, \emph{Weak solutions of linear partial differential equations}, J. Math. Pure Appl.  \textbf{35} (1956), 193-202.
		\bibitem{Boo}
		W.~Boothby, \emph{An introduction to differentiable manifolds and Riemannian geometry},
		Pure and Applied Mathematics Series \textbf{120}, Gulf Professional Publishing, 2003. 
	
 \bibitem{Car}
 L.~Carleson, \emph{Selected problems on exceptional sets}, Van Nostrand, Princeton, N. J.,  1967.
		\bibitem{Ch}
		E.~Chirka, \emph{Complex analytic sets},
		Mathematics and its Applications (Soviet Series),  \textbf{46},
		Kluwer Academic Publishers Group, Dordrecht, 1989.

		\bibitem{Con}
        J.~Conway, \emph{A course in operator theory}, GTM \textbf{21}, American Mathematical Society, Rhode Island, 2000.
		\bibitem{CW}
		C.~Cowen and R.~Wahl, \emph{Commutants of finite Blaschke product multiplication operators}, Invariant subspaces of the shift operator, 99-114, Contemp. Math., 638, Centre Rech. Math. Proc., Amer. Math. Soc., Providence, RI, 2015.
		\bibitem{Da}
		K.~Davidson, \emph{C*-algebras by example}, Fields Institute Monographs, \textbf{6},
		AMS, Rhode Island, 1996.
		\bibitem{DSe}
 G.~Dini, A.~Selvaggi Primicerio, \emph{Proper holomorphic mappings
  between generalized pseudoellipsoids},
   Ann. Mat. Pura Appl. (4)158(1991), 219-229.
		\bibitem{DH}
		H.~Dan and H.~Huang, \emph{Multiplication operators defined by a class of  polynomials on} $L_a^2(\mathbb{D}^2)$,
		Integr. Equ. Oper. Theory, \textbf{80}(2014), 581-601. 
		
		
		\bibitem{DPW}
		R.~Douglas, M.~Putinar and K.~Wang, \emph{Reducing subspaces for analytic multipliers of the Bergman space},
		J. Funct. Anal.  \textbf{263}(2012), 1744-1765.  
		\bibitem{DSZ}
		R.~Douglas, S.~Sun and D.~Zheng, \emph{Multiplication operators on the Bergman space via analytic continuation},
		Adv. Math. \textbf{226}(2011), 541-583.
		
		\bibitem{Gh}
		G.~Ghosh,   \emph{Multiplication operator on the Bergman space by a proper map},
J. Math. Anal. Appl. 510 (2022), no. 2, Paper No. 126026, 12 pp. 
		\bibitem{GH1}
		K.~Guo and H.~Huang,  \emph{On multiplication operators of the Bergman space:
			Similarity, unitary equivalence and reducing subspaces},   J. Operator Theory, \textbf{65}(2011), 355-378. 
		\bibitem{GH2}
		K.~Guo and H.~Huang, \emph{Multiplication operators defined by covering maps on the
			Bergman space: the connection between operator theory and von Neumann   algebras}, J. Funct. Anal. \textbf{260}(2011), 1219-1255.
		
		\bibitem{GH3} %
		K.~Guo and H.~Huang, \emph{Geometric constructions of thin Blaschke products and reducing subspace problem},
		Proc. London Math. Soc.  \textbf{109}(2015), 1050-1091. 
		\bibitem{GH4}
		K.~Guo and H.~Huang, \emph{Multiplication operators on the Bergman space}, Lecture Notes in Mathematics \textbf{2145},
		Springer, Heidelberg, 2015.
		
		\bibitem{GW}
		K.~Guo and X.~Wang, \emph{Reducing subspaces of tensor products of weighted shifts},
		Sci. China Ser. A. \textbf{59}(2016), 715-730. 
		\bibitem{HZ1}
		H.~Huang and D.~Zheng, \emph{Multiplication operators on the Bergman spaces of polygons},
		J. Math. Anal. Appl. 456(2017), no. 2, 1049-1061.
		\bibitem{HZ2}
		H.~Huang and D.~Zheng, \emph{Unitary equivalence of multiplication operators on the Bergman spaces of polygons.}, Canad. Math. Bull. 65 (2022), no. 1, 123-133.
		
		
		%
		\bibitem{Kr}
		S.~Krantz,  \emph{Function theory of several complex variables},
		Pure and Applied Mathematics, A Wiley-Interscience Publication, Inc., New York, 1982.
		\bibitem{Mi}
		J.~Milnor, \emph{Dynamics in one complex variable},  Princeton
		University Press, Ann. of   Math. Stud. \textbf{160}, 2006. 
		\bibitem{Ran} R.~ Range,  \emph{ Holomorphic functions and integral representations in several complex variables}, Graduate Texts in Mathematics, 108,  Springer-Verlag, New York, 1986.
		\bibitem{Re1}
		R.~Remmert, \emph{Projektionen analytischer Mengen}, Math. Ann. \textbf{130}(1956),
		410-441.
		\bibitem{Re2}
		R.~Remmert, \emph{Holomorphe und meromorphe Abbildungen komplexer R$\ddot{a}$ume}
		Math. Ann. \textbf{133}, 1957, 328-370.
		\bibitem{Ru1}
		W.~Rudin, \emph{Function theory in polydiscs},  Benjamin, New York,
		1969.
		\bibitem{Ru2}
		W.~Rudin, \emph{Function theory in the unit ball of $\mathbb{C}^n$},
		Grundlehren der Math. \textbf{241}, Springer, New York, 1980.
		
		\bibitem{Ru3}
		W.~Rudin, \emph{Real and complex analysis}, 3rd edition, McGraw-Hill Book Co., New York, 1987.
		
		\bibitem{SZZ}
		S.~Sun, D.~Zheng and C.~Zhong, \emph{Classification of reducing
			subspaces of a  class of multiplication operators via the Hardy
			space of the bidisk},  Canad. J. Math. \textbf{62}(2010), 415-438.
		\bibitem{T1}
		J.~Thomson, \emph{The commutant of a class of analytic Toeplitz
			operators}, Amer. J. Math. \textbf{99}(1977), 522-529.
		\bibitem{T2}
		J.~Thomson,
		\emph{The commutant of a class of analytic Toeplitz operators II},
		Indiana Univ. Math. J. \textbf{25}(1976), 793-800.
		\bibitem{Ti}
		A.~Tikaradze,  \emph{Multiplication operators  on the Bergman spaces of pseudoconvex domains},
		New York J. Math. 21(2015), 1327-1345.
		\bibitem{WDH}
		X.~Wang, H.~Dan and H.~Huang,
		\emph{Reducing subspaces of multiplication operators with the symbol  $\alpha z^k+\beta w^l$   on $L_a^2(\mathbb{D}^2)$},
		Sci. China Ser. A.  \textbf{58}(2015), 1-14. 
		\bibitem{Zhu}
		K.~Zhu,   \emph{Reducing subspaces for  a class of multiplication operators},
		J. London Math. Soc.,  \textbf{62}(2000),  553-568.
		\bibitem{Zhub}
		K.~Zhu,   \emph{An Introduction to Operator Algebras},
		Stud. Adv. Math., CRC Press, Boca Raton, FL, 1993. x+157 pp.
	}
\end{thebibliography}
\end{document}